\documentclass[12pt]{article}
\usepackage{latexsym,amsmath,amssymb,amsfonts,amsthm,bbm,mathrsfs,stmaryrd,color,breakcites,dsfont}
\usepackage{natbib}
\usepackage{dsfont,url}
\usepackage{enumerate}
\RequirePackage[colorlinks,citecolor=blue,urlcolor=blue,breaklinks]{hyperref}
\usepackage{graphicx}
\usepackage{graphics}
\usepackage{psfrag}
\usepackage{caption}
\usepackage{comment}
\usepackage{titling}

\newtheorem{theorem}{Theorem}
\newtheorem{lemma}[]{Lemma}
\newtheorem{example}[]{Example}
\newtheorem{proposition}[lemma]{Proposition}
\newtheorem{corollary}[theorem]{Corollary}

\DeclareMathOperator*{\argmin}{argmin}

\DeclareMathOperator*{\sargmin}{sargmin}

\begin{document}
	
\title{Classification with imperfect training labels}
\author{Timothy I. Cannings$^\ast$, Yingying Fan$^\dag$ and Richard J. Samworth$^\ddag$ \\  \textit{$^\ast$University of Edinburgh, $^\dag$University of Southern California} \\ and \textit{$^\ddag$University of Cambridge}}
	\date{}
	\maketitle
	\begin{abstract}
	We study the effect of imperfect training data labels on the performance of classification methods.   In a general setting, where the probability that an observation in the training dataset is mislabelled may depend on both the feature vector and the true label, we bound the excess risk of an arbitrary classifier trained with imperfect labels in terms of its excess risk for predicting a noisy label.  This reveals conditions under which a classifier trained with imperfect labels remains consistent for classifying uncorrupted test data points.  Furthermore, under stronger conditions, we derive detailed asymptotic properties for the popular $k$-nearest neighbour ($k$nn), support vector machine (SVM) and linear discriminant analysis (LDA) classifiers.  One consequence of these results is that the $k$nn and SVM classifiers are robust to imperfect training labels, in the sense that the rate of convergence of the excess risks of these classifiers remains unchanged; in fact, our theoretical and empirical results even show that in some cases, imperfect labels may improve the performance of these methods.  On the other hand, the LDA classifier is shown to be typically inconsistent in the presence of label noise unless the prior probabilities of each class are equal.  Our theoretical results are supported by a simulation study.
	\end{abstract}

\section{Introduction}
Supervised classification is one of the fundamental problems in statistical learning.  In the basic, binary setting, the task is to assign an observation to one of two classes, based on a number of previous training observations from each class.   Modern applications include, among many others, diagnosing a disease using genomics data \citep{Wright:2015}, determining a user's action from smartphone telemetry data \citep{Lara:2013}, and detecting fraud based on historical financial transactions \citep{Bolton:2002}. 

In a classification problem it is often the case that the class labels in the training data set are inaccurate.  For instance, an error could simply arise due to a coding mistake when the data were recorded.  In other circumstances, such as the disease diagnosis application mentioned above, errors may occur due to the fact that, even to an expert, the true labels are hard to determine, especially if there is insufficient information available.  Moreover, in modern Big Data applications with huge training data sets, it may be impractical and expensive to determine the true class labels, and as a result the training data labels are often assigned by an imperfect algorithm.  Services such as the \emph{Amazon Mechanical Turk}, (see~\url{https://www.mturk.com}), allow practitioners to obtain training data labels relatively cheaply via crowdsourcing.  Of course, even after aggregating a large crowd of workers' labels, the result may still be inaccurate.  \citet{Chen:2015} and \citet{Zhang:2016} discuss crowdsourcing in more detail, and investigate strategies for obtaining the most accurate labels given a cost constraint.  

The problem of label noise was first studied by \citet{Lachenbruch:1966}, who investigated the effect of imperfect labels in two-class linear discriminant analysis.  Other early works of note include \citet{Lachenbruch:1974}, \citet{Angluin:1988} and \citet{Lugosi:1992}.  

\citet{Frenay:2014a} and \citet{Frenay:2014b} provide recent overviews of work on the topic. In the simplest, homogeneous setting, each observation in the training dataset is mislabelled independently with some fixed probability.  \citet{vanRooyen:2015a} study the effects of homogeneous label errors on the performance of empirical risk minimization (ERM) classifiers, while \citet{Long:2010} consider boosting methods in this same homogeneous noise setting.  Other recent works focus on class-dependent label noise, where the probability that a training observation is mislabelled depends on the true class label of that observation;  see \citet{Stempfel:2009}, \citet{Natarajan:2013}, \citet{Scott:2013}, \citet{Blanchard:2016}, \citet{Liu:2016} and \citet{Patrini:2016}.  An alternative model assumes the noise rate depends on the feature vector of the observation.  \citet{Manwani:2013} and \citet{Ghosh:2015} investigate the properties of ERM classifiers in this setting; see also \citet{Awasthi:2015}.  \citet{Menon:2016} propose a \textit{generalized boundary consistent} label noise model, where observations near the optimal decision boundary are more likely to be mislabelled, and study the effects on the properties of the receiver operator characteristics curve. 

In the more general setting, where the probability of mislabelling is both feature- and class-dependent, \citet{Bootkrajang:2012,Bootkrajang:2014} and \citet{Bootkrajang:2016} study the effect of label noise on logistic regression classifiers, while \citet{Li:2017}, \citet{Patrini:2017} and \citet{Rolnick:2017} consider neural network classifiers.   On the other hand, \citet{Cheng:2017} investigate the performance of an ERM classifier in the feature- and class-dependent noise setting when the true class conditional distributions have disjoint support.  

Our first goal in the present paper is to provide general theory to characterize the effect of feature- and class-dependent heterogeneous label noise for an arbitrary classifier.  We first specify general conditions under which the optimal prediction of a true label and a noisy label are the same for every feature vector.  Then, under slightly stronger conditions, we relate the misclassification error when predicting a true label to the corresponding error when predicting a noisy label.  More precisely, we show that the excess risk, i.e.~the difference between the error rate of the classifier and that of the optimal, Bayes classifier, is bounded above by the excess risk associated with predicting a noisy label multiplied by a constant factor that does not depend on the classifier used; see Theorem~\ref{thm:hetnoise}.  Our results therefore provide conditions under which a classifier trained with imperfect labels remains consistent for classifying uncorrupted test data points.

As applications of these ideas, we consider three popular approaches to classification problems, namely the $k$-nearest neighbour ($k$nn), support vector machine (SVM) and linear discriminant analysis (LDA) classifiers.   In the perfectly labelled setting, the $k$nn classifier is consistent for any data generating distribution and the SVM classifier is consistent when the distribution of the feature vectors is compactly supported.  Since the label noise does not change the marginal feature distribution, it follows from our results mentioned in the previous paragraph that these two methods are still consistent when trained with imperfect labels that satisfy our assumptions, which, in the homogeneous noise case, even allow up to 1/2 of the training data to be labelled incorrectly.  On the other hand, for the LDA classifier with Gaussian class-conditional distributions, we derive the asymptotic risk in the homogeneous label noise case.  This enables us to deduce that the LDA classifier is typically not consistent when trained with imperfect labels, unless the class prior probabilities are equal to 1/2.

Our second main contribution is to provide greater detail on the asymptotic performance of the $k$nn and SVM classifiers in the presence of label noise, under stronger conditions on the data generating mechanism and noise model.  In particular, for the $k$nn classifier, we derive the asymptotic limit for the ratio of the excess risks of the classifier trained with imperfect and perfect labels, respectively.   This reveals the nice surprise that using imperfectly-labelled training data can in fact improve the performance of the $k$nn classifier in certain circumstances. To the best of our knowledge, this is the first formal result showing that label noise can help with classification.  For the SVM classifier, we provide conditions under which the rate of convergence of the excess risk is unaffected by label noise, and show empirically that this method can also benefit from label noise in some cases. 

In several respects, our theoretical analysis acts a counterpoint to the folklore in this area.  For instance, \citet{Okamoto:1997} analysed the performance of the $k$nn classifier in the presence of label noise.  They considered relatively small problem sizes and small values of~$k$, where the $k$nn classifier performs poorly when trained with imperfect labels; on the other hand, our Theorem~\ref{thm:knnhet} reveals that for larger values of $k$, which diverge with $n$, the asymptotic effect of label noise is relatively modest, and may even improve the performance of the classifier.   As another example, \citet{Manwani:2013} and \citet{Ghosh:2015} claim that SVM classifiers perform poorly in the presence of label noise; our Theorem~\ref{thm:SVMhet} presents a different picture, however, at least as far as the rate of convergence of the excess risk is concerned.  Finally, in two-class Gaussian discriminant analysis, \citet{Lachenbruch:1966} showed that LDA is robust to homogeneous label noise when the two classes are equally likely \citep[see also][Section III-A]{Frenay:2014b}. We observe, though, that this robustness is very much the exception rather than the rule: if the prior probabilities are not equal, then the LDA classifier is almost invariably not consistent when trained with imperfect labels; cf.~Theorem~\ref{thm:LDAhomo}.

Although it is not the focus of this paper, we mention briefly that another line of work on label noise investigates techniques for identifying mislabelled observations and either relabelling them, or simply removing them from the training data set. Such methods are sometimes referred to as data \textit{cleansing} or \textit{editing} techniques;  see for instance \citet{Wilson:1972}, \citet{Wilson:2000} and \citet{Cheng:2017}; as well as \citet[Section~3.2]{Frenay:2014a}, who provide a general overview of popular methods for editing training data sets.  Other authors focus on estimating the noise rates and recovering the clean class-conditional distributions  \citep{Blanchard:2016,Northcutt:2017}. 

The remainder of this paper is organized as follows.  In Section~\ref{sec:setting} we introduce our general statistical setting, while in Section~\ref{sec:finite}, we present bounds on the excess risk of an arbitrary classifier trained with imperfect labels under very general conditions.  In Section~\ref{sec:asymptotic}, we derive the asymptotic properties of the $k$nn, SVM and LDA classifiers when trained with noisy labels.  Our empirical experiments, given in Section~\ref{sec:sims}, show that this asymptotic theory corresponds well with finite-sample performance. Finally in the appendix we present the proofs underpinning our theoretical results, as well as an illustrative example involving the 1-nearest neighbour classifier.

The following notation is used throughout the paper.  We write $\|\cdot\|$ for the Euclidean norm on $\mathbb{R}^d$, and for $r > 0$ and $z \in \mathbb{R}^d$, write $B_z(r) = \{x \in \mathbb{R}^d:\|x-z\| < r\}$ for the open Euclidean ball of radius $r$ centered at $z$, and let $a_d = \pi^{d/2}/\Gamma(1+d/2)$ denote the $d$-dimensional volume of $B_0(1)$.  If $A \in \mathbb{R}^{d \times d}$, we write $\|A\|_{\mathrm{op}}$ for its operator norm.  For a sufficiently smooth real-valued function $f$ defined on $D \subseteq \mathbb{R}^m$, and for $x \in D$, we write $\dot{f}(x) = (f_1(x),\ldots,f_m(x))^T$ and $\ddot{f}(x) = (f_{jk}(x))_{j,k=1}^m$ for its gradient vector and Hessian matrix at $x$ respectively.  Finally, we write $\triangle$ for symmetric difference, so that $\mathcal{A} \triangle \mathcal{B} = (\mathcal{A}^c \cap \mathcal{B}) \cup (\mathcal{A} \cap \mathcal{B}^c)$.  

We conclude this section with a preliminary study to demonstrate our new results for the $k$nn, SVM and LDA classifiers in the homogeneous noise case.

\begin{example}
\label{Ex:prelim}
\emph{In this motivating example, we demonstrate the surprising effects of imperfect labels on the performance of the $k$nn, SVM and LDA classifiers.  We generate $n$ independent training data pairs, where the prior probabilities of classes 0 and 1 are $9/10$ and $1/10$ respectively; class 0 and 1 observations have bivariate normal distributions with means $\mu_0 = (-1,0)^T$ and $\mu_1 = (1,0)^T$ respectively, and common identity covariance matrix.  We then introduce label noise in the training data set by flipping the true training data labels independently with probability $\rho =$ 0.3.  One example of a data set of size $n=1000$ from this model, both before and after label noise is added, is shown in Figure~\ref{fig:dataplot}.}

\begin{figure}[ht!]
	\centering
	\includegraphics[width=\textwidth]{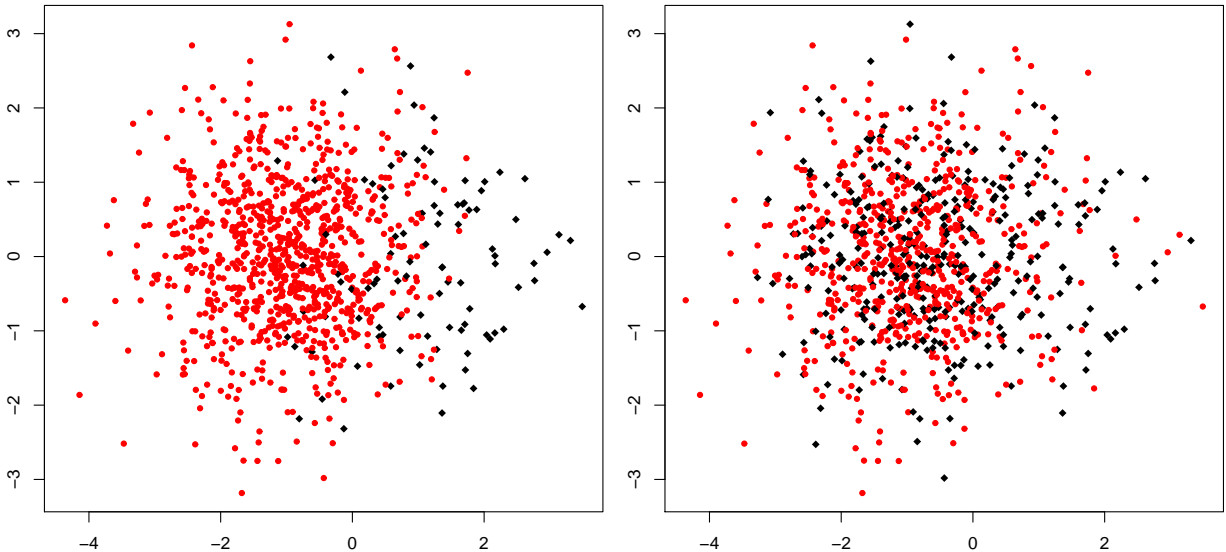} \caption{{\small One training dataset from the model in Example~\ref{Ex:prelim} for $n = 1000$, without label noise (left) and with label noise (right). We plot class~0 in red and class~1 in black.}}
	\label{fig:dataplot}
\end{figure}

\begin{figure}
	\centering
\includegraphics[width=0.6\textwidth]{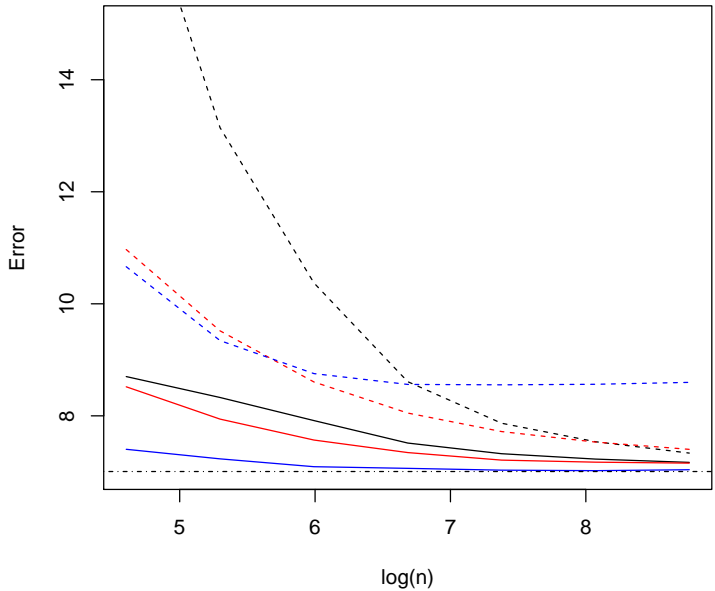}
\caption{{\small Risks ($\%$) of the $k$nn (black), SVM (red) and LDA (blue) classifiers trained using perfect (solid lines) and imperfect labels (dotted lines).  The dot-dashed line shows the Bayes risk, which is 7.0\%.}}
\label{fig:errorplot}
\end{figure}

\emph{In Figure~\ref{fig:errorplot},  we present the percentage error rates, both with and without label noise, of the $k$nn, SVM and LDA classifiers.  The error rates were estimated by the average over 1000 repetitions of the experiment of the percentage of misclassified observations on a test set, without label noise, of size 1000.   We set $k = k_{n} = \lfloor n^{2/3}/2 \rfloor$ for the $k$nn classifier, and set the tuning parameter $\lambda = 1$ for the SVM classifier; see~\eqref{eq:SVMK}.}

\emph{In this simple setting where the decision boundary of the Bayes classifier is a hyperplane, all three classifiers perform very well with perfectly labelled training data, especially LDA, whose derivation was motivated by Gaussian class-conditional distributions with common covariance matrix.  With mislabelled training data, the performance of all three classifiers is somewhat affected, but the $k$nn and SVM classifiers are relatively robust to the label noise, particularly for large $n$.  Indeed, we will show that these classifiers remain consistent in this setting.  The gap between the performance of the LDA classifier and that of the Bayes classifier, however, persists even for large $n$; this again is in line with our theory developed in Theorem~\ref{thm:LDAhomo}, where we derive the asymptotic risk of the LDA classifier trained with homogeneous label errors.  The limiting risk is given explicitly in terms of the noise rate $\rho$, the prior probabilities, and the Mahalanobis distance between the two class-conditional distributions.}
\end{example}

\section{Statistical setting}
\label{sec:setting} 
Let $\mathcal{X}$ be a measurable space.  In the basic binary classification problem, we observe independent and identically distributed training data pairs $(X_1, Y_1), \ldots, (X_n, Y_n)$ taking values in $\mathcal{X} \times \{0,1\}$ with joint distribution~$P$.  The task is to predict the class $Y$ of a new observation $X$, where $(X,Y) \sim P$ is independent of the training data. 

Define the \emph{prior probabilities} $\pi_1 = \mathbb{P}(Y = 1) = 1- \pi_0 \in (0,1)$ and \emph{class-conditional distributions} $X \mid \{Y=r\} \sim P_r$ for $r=0,1$.  The \emph{marginal feature distribution} of $X$ is denoted~$P_X$ and we define the \emph{regression function} $\eta(x) = \mathbb{P}\bigl( Y = 1 \mid X = x)$.   A \emph{classifier} $C$ is a measurable function from $\mathcal{X}$ to $\{0,1\}$, with the interpretation that a point $x \in \mathcal{X}$ is assigned to class~$C(x)$.   

The \emph{risk} of a classifier $C$ is $R(C) = \mathbb{P}\{C(X) \neq Y\}$; it is minimized by the \emph{Bayes classifier}
\[
C^{\mathrm{Bayes}}(x) = \left\{ \begin{array}{ll} 1& \text{\quad if } \eta(x) \geq 1/2
\\ 0 & \text{\quad otherwise.} \end{array} \right.
\]
However, since $\eta$ is typically unknown, in practice we construct a classifier $C_{n}$, say, that depends on the $n$ training data pairs.  We say $(C_n)$ is \emph{consistent} if $R(C_{n}) - R(C^{\mathrm{Bayes}}) \rightarrow 0$ as $n \rightarrow \infty$.  When we write $R(C_n)$ here, we implicitly assume that $C_{n}$ is a measurable function from $(\mathcal{X} \times \{0,1\})^n\times\mathcal{X}$ to $\{0,1\}$, and the probability is taken over the joint distribution of $(X_1,Y_1),\ldots,(X_n,Y_n), (X,Y)$.   It is convenient to set $\mathcal{S} = \{x \in \mathcal{X}:\eta(x) = 1/2\}$.

In this paper, we study settings where the true class labels $Y_1, \ldots, Y_n$ for the training data are not observed.  Instead we see $\tilde{Y}_1, \ldots, \tilde{Y}_n$, where the noisy label $\tilde{Y}_{i}$ still takes values in $\{0,1\}$, but may not be the same as $Y_i$.  The task, however, is still to predict the true class label $Y$ associated with the test point $X$.   We can therefore consider an augmented model where $(X, Y, \tilde{Y}), (X_1, Y_1, \tilde{Y}_1), \ldots, (X_n, Y_n, \tilde{Y}_n)$ are independent and identically distributed triples taking values in $\mathcal{X} \times \{0,1\} \times \{0,1\}$.  

At this point the dependence between $Y$ and $\tilde{Y}$ is left unrestricted, but we introduce the following notation:  define measurable functions $\rho_0, \rho_1 : \mathcal{X} \rightarrow [0,1]$ by $\rho_{r}(x) = \mathbb{P}(\tilde{Y} \neq Y \mid X = x, Y = r)$.  Thus, letting $Z \mid \{X = x, Y = r\} \sim \mathrm{Bin}(1, 1- \rho_r(x))$ for $r=0,1$, we can write $\tilde{Y} =  Z Y+ (1-Z)(1-Y)$. We refer to the case where $\rho_0(x) = \rho_1(x) = \rho$ for all $x \in \mathcal{X}$ as \emph{$\rho$-homogeneous noise}.  Further, let $\tilde{P}$ denote the joint distribution of $(X,\tilde{Y})$, and let $\tilde{\eta}(x) = \mathbb{P}(\tilde{Y} = 1 \mid X = x)$ denote the regression function for $\tilde{Y}$, so that
\begin{align}
\label{eq:tildeeta}
\tilde{\eta}(x) &= \eta(x) \mathbb{P}(\tilde{Y} = 1 \mid X = x, Y = 1) + \{1- \eta(x)\}\mathbb{P}(\tilde{Y} = 1 \mid X = x, Y = 0) \nonumber
\\& = \eta(x) \{1 - \rho_1(x)\} + \{1 - \eta(x)\}\rho_0(x).
\end{align}
We also define the \textit{corrupted Bayes classifier}
\[
\tilde{C}^{\mathrm{Bayes}}(x) = \left\{ \begin{array}{ll} 1& \text{\quad if } \tilde{\eta}(x) \geq 1/2
\\ 0 & \text{\quad otherwise,} \end{array} \right.
\]
which minimizes the \textit{corrupted risk} $\tilde{R}(C) = \mathbb{P}\{C(X) \neq \tilde{Y}\}$.

\section{Excess risk bounds for arbitrary classifiers}
\label{sec:finite}
A key property in this work will be that the Bayes classifier is preserved under label noise; more specifically, in Theorem~\ref{thm:hetnoise}(i) below, we will provide conditions under which  
\begin{equation}
\label{eq:symmetric}
P_{X}\bigl( \{x \in \mathcal{S}^c : \tilde{C}^{\mathrm{Bayes}}(x) \neq C^{\mathrm{Bayes}}(x)\} \bigr) = 0.
\end{equation}
In Theorem~\ref{thm:hetnoise}(ii), we go on to show that, under slightly stronger conditions on the label error probabilities and for an arbitrary classifier $C$, we can bound the excess risk $R(C)  - R(C^{\mathrm{Bayes}})$ of predicting the true label by a multiple of the excess risk of predicting a noisy label $\tilde{R}(C) - \tilde{R}(\tilde{C}^{\mathrm{Bayes}})$, where this multiple does not depend on the classifier $C$.   This latter result is particularly useful when the classifier $C$ is trained using the imperfect  labels, that is with the training data $(X_{1}, \tilde{Y}_{1}), \ldots, (X_{n}, \tilde{Y}_{n})$, because, as will be shown in the next section, we are able to provide further control of $\tilde{R}(C) - \tilde{R}(\tilde{C}^{\mathrm{Bayes}})$ for specific choices of $C$.   

It is convenient to let $\mathcal{B} = \{x \in \mathcal{S}^c : \rho_0 (x) + \rho_1(x) < 1 \}$, and let
\[
\mathcal{A} =\biggl\{x \in \mathcal{B}:\frac{\rho_1(x) - \rho_0(x)}{\{2\eta(x)-1\}\{1 - \rho_0(x) - \rho_1(x)\}} < 1\biggr\}.
\]
\begin{theorem}
	\label{thm:hetnoise}
	(i) We have 
	\begin{equation}
	\label{eq:noisecond}
	P_X\bigl(\mathcal{A} \, \triangle \, \{x \in \mathcal{B}: \tilde{C}^{\mathrm{Bayes}}(x) = C^{\mathrm{Bayes}}(x) \} \bigr) = 0.
	\end{equation}
In particular, if $P_{X}(\mathcal{A}^c \cap \mathcal{S}^c) = 0$, then~\eqref{eq:symmetric} holds.
	
	(ii) Now suppose, in fact, that there exist $\rho^* < 1/2$ and $a^* < 1$ such that 
	$P_{X}(\{x \in \mathcal{S}^c : \rho_0 (x) + \rho_1(x) > 2\rho^* \}) = 0$, and 
	\[
	P_X \biggl(\biggl\{ x\in \mathcal{B}: \frac{\rho_1(x) - \rho_0(x)}{\{2\eta(x)-1\}\{1 - \rho_0(x) - \rho_1(x)\}} > a^* \biggr\}\biggr) = 0.
	\]
	Then, for any classifier $C$,
	\[
	R(C)  - R(C^{\mathrm{Bayes}})  \leq  \frac{\tilde{R}(C) - \tilde{R}(\tilde{C}^{\mathrm{Bayes}})}{(1- 2\rho^*)(1-a^*) }.
	\]
\end{theorem} 
In Theorem~\ref{thm:hetnoise}(i), the condition $P_X(\mathcal{A}^c \cap \mathcal{S}^c) = 0$ restricts the difference between the two mislabelling probabilities at $P_X$-almost all $x \in \mathcal{S}^c$, with stronger restrictions where $\eta(x)$ is close to $1/2$ and where $\rho_{0}(x) + \rho_1(x)$ is close to 1.  Moreover, since $\mathcal{A}\subseteq\mathcal{B}$, we also have $P_X(\mathcal{B}^c \cap \mathcal{S}^c) = 0$, which limits the total amount of label noise at each point; cf.~\citet[Assumption~1]{Menon:2016}.  In particular, it ensures that
\[
\mathbb{P}(\tilde{Y} \neq Y \mid X = x) = \eta(x) \rho_{1}(x) + \{1 - \eta(x)\} \rho_{0}(x) < 1, 
\] 
for $P_{X}$-almost all $x \in \mathcal{S}^c$.  In part~(ii), the requirement on $a^*$ imposes a slightly stronger restriction on the same weighted difference between the two mislabelling probabilities compared with part~(i).

The conditions in Theorem~\ref{thm:hetnoise} generalize those given in the existing literature by allowing a wider class of noise mechanisms.  For instance, in the case of $\rho$-homogeneous noise, we have $P_X(\mathcal{A}^c \cap \mathcal{S}^c) = 0$ provided only that $\rho < 1/2$.  In fact, in this setting, we may take $a^* = 0$ \citep[Theorem~1]{Ghosh:2015}.  More generally, we may also take $a^{*} = 0$ if the noise depends only on the feature vector and not the true class label, i.e.~$\rho_{0}(x) = \rho_{1}(x)$ for all $x$ \citep[Proposition~4]{Menon:2016}. 

The proof of Theorem~\ref{thm:hetnoise}(ii) relies on the following proposition, which provides a bound on the excess risk for predicting a true label, assuming only that~\eqref{eq:symmetric} holds.  
\begin{proposition}
\label{prop:prelim}
Assume that \eqref{eq:symmetric} holds. Further, for $\kappa > 0$, let 
\[
A_{\kappa} = \Bigl\{x \in \mathcal{X} : |2\eta(x) - 1| \leq  \kappa |2\tilde{\eta}(x) - 1|\Bigr\}.
\]
 Then, for any classifier $C$, 
\begin{equation}
\label{eq:minbound}
R(C)\! - \!R(C^{\mathrm{Bayes}}) \leq \min\Bigl[\mathbb{P}\{C(X)\! \neq\! \tilde{C}^{\mathrm{Bayes}}(X)\}, \inf_{\kappa > 0} \Bigl\{\kappa \{\tilde{R}(C) \!-\! \tilde{R}(\tilde{C}^{\mathrm{Bayes}})\} \!+\! P_X(A_{\kappa}^c) \Bigr\} \Bigr].
\end{equation} 
\end{proposition}
Our main focus in this work is on settings where $\tilde{C}^{\mathrm{Bayes}}$ and $C^{\mathrm{Bayes}}$ agree, i.e.~\eqref{eq:symmetric} holds, because this is where we can hope for classifiers to be robust to label noise.  However, in this instance, we present a more general version of Proposition~\ref{prop:prelim} as Proposition~\ref{prop:prelim2} in the appendix; this bounds the excess risk of an arbitrary classifier without the assumption that~\eqref{eq:symmetric} holds.  We see in that result, there is an additional contribution to the risk bound of $R(\tilde{C}^{\mathrm{Bayes}}) - R(C^{\mathrm{Bayes}}) \geq 0$.  See also, for instance, \citet{Natarajan:2013}, who study asymmetric homogeneous noise, where $\rho_0(x) = \rho_0 \neq \rho_1 = \rho_1(x)$, with $\rho_0$ and $\rho_1$ known.

We can regard $|2\eta(x) - 1|$ as a measure of the ease of classifying $x$.  Hence, in Proposition~\ref{prop:prelim}, we can interpret $A_{\kappa}$ as the set of points $x$ where the relative difficulty of classifying $x$ in the corrupted problem compared with its uncorrupted version is controlled.  The level of this control can then be traded off against the measure of the exceptional set $A_\kappa^c$.  

To provide further understanding of Proposition~\ref{prop:prelim}, observe that in general, we have
\begin{align*}
\tilde{R}(C) - \tilde{R}(\tilde{C}^{\mathrm{Bayes}}) & = \int_{\mathcal{X}} \bigl[\mathbb{P}\{C(x) = 0\} - \mathbbm{1}_{\{\tilde{\eta}(x) < 1/2\}}\bigr] \{2\tilde{\eta}(x) - 1\} \, dP_{X}(x)
\\ & \leq \mathbb{P}\{C(X) \neq \tilde{C}^{\mathrm{Bayes}}(X)\}. 
\end{align*}
Thus, if $P_X(A_{1}^c) = 0$, then the second term in the minimum in~\eqref{eq:minbound} gives a better bound than the first.  However, typically in practice, we would have that $P_X(A_{1}^c) \neq 0$, and indeed, in Example~\ref{ex:1nn} in the appendix, we show that for the 1-nearest neighbour classifier with homogeneous noise, either of the two terms in the minimum in \eqref{eq:minbound} can be smaller, depending on the noise level.  As a consequence of Proposition~\ref{prop:prelim}, we have the following corollary.
\begin{corollary}
\label{cor:consistent}
 Suppose that $(\tilde{C}_n)$ is a sequence of classifiers satisfying $\tilde{R}(\tilde{C}_n) \rightarrow \tilde{R}(\tilde{C}^{\mathrm{Bayes}})$ and assume that \eqref{eq:symmetric} holds.  Further, let $\tilde{\mathcal{S}} = \{x \in \mathcal{X}:\tilde{\eta}(x) = 1/2\}$.  Then
  \[
    \limsup_{n \rightarrow \infty} R(\tilde{C}_n) - R(C^{\mathrm{Bayes}}) \leq P_X(\tilde{\mathcal{S}} \setminus \mathcal{S}).
  \]
In particular, if $P_X(\tilde{\mathcal{S}} \setminus \mathcal{S}) = 0$, then $R(\tilde{C}_n) \rightarrow R(C^{\mathrm{Bayes}})$ as $n \rightarrow \infty$.  
\end{corollary}
The condition $\tilde{R}(\tilde{C}_n) \rightarrow \tilde{R}(\tilde{C}^{\mathrm{Bayes}})$ asks that the classifier is consistent for predicting a corrupted test label.  In Section~\ref{sec:asymptotic} we will see that appropriate versions of the corrupted $k$nn and SVM classifiers satisfy this condition, provided, in the latter case, that the feature vectors have compact support.  To understand the strength of Corollary~\ref{cor:consistent}, consider the special case of $\rho$-homogeneous noise, and a classifier $\tilde{C}_n$ that is consistent for predicting a noisy label when trained with corrupted data.  Then $\tilde{\mathcal{S}} = \mathcal{S}$ by~\eqref{eq:tildeeta}, so provided only that $\rho < 1/2$, Corollary~\ref{cor:consistent} ensures that $\tilde{C}_n$ remains consistent for predicting a true label when trained using the corrupted data.   

\section{Asymptotic properties} 
\label{sec:asymptotic}
\subsection{The $k$-nearest neighbour classifier}
\label{sec:knn}
We now specialize to the case $\mathcal{X} = \mathbb{R}^d$.  The $k$nn classifier assigns the test point $X$ to a class based on a majority vote over the class labels of the $k$ nearest points among the training data.  More precisely, given $x \in \mathbb{R}^d$, let $(X_{(1)}, Y_{(1)}), \ldots, (X_{(n)}, Y_{(n)})$ be the reordering of the training data pairs such that
\[
\|X_{(1)} - x \| \leq \ldots \leq \|X_{(n)} - x \|,
\]
where ties are broken by preserving the original ordering of the indices.  For $k \in \{1, \ldots, n\}$, the \emph{$k$-nearest neighbour classifier} is
\[
C^{k\mathrm{nn}}(x) = C_n^{k\mathrm{nn}}(x) = \left\{ \begin{array}{ll} 1& \text{\quad if } \frac{1}{k} \sum_{i = 1} ^{k} \mathbbm{1}_{\{Y_{(i)} =1 \} } \geq 1/2
\\ 0 & \text{\quad otherwise.} \end{array} \right.
\]

This simple and intuitive method has received considerable attention since it was introduced by \citet{Fix:1951,Fix:1989}.  \citet{Stone:1977} showed that the $k$nn classifier is universally consistent, i.e., $R(C^{k\mathrm{nn}}) \rightarrow R(C^{\mathrm{Bayes}})$ for any distribution $P$, as long as $k = k_n \rightarrow \infty$ and $k/n \rightarrow 0$ as $n \rightarrow \infty$.    For a substantial overview of the early work on the theoretical properties of the $k$nn classifier, see \citet{PTPR:1996}.  Further recent studies include \citet{Kulkarni:1995}, \citet{Audibert:2007}, \citet{Hall:2008}, \citet{Biau:2010}, \citet{Samworth:2012}, \citet{Chaudhuri:2014}, \citet{Gadat:16}, \citet{Celisse:2018} and \citet{CBS:2017}.  

Here we study the properties of the \emph{corrupted $k$-nearest neighbour classifier}
 \[
\tilde{C}^{k\mathrm{nn}}(x) = \tilde{C}_n^{k\mathrm{nn}}(x)  = \left\{ \begin{array}{ll} 1& \text{\quad if } \frac{1}{k}  \sum_{i = 1} ^{k} \mathbbm{1}_{\{\tilde{Y}_{(i)} =1 \} } \geq 1/2
\\ 0 & \text{\quad otherwise,} \end{array} \right.
\]
where $\tilde{Y}_{(i)}$ denotes the corrupted label of $(X_{(i)},Y_{(i)})$.  Since the $k$nn classifier is universally consistent, we have $\tilde{R}(\tilde{C}^{k\mathrm{nn}}) \rightarrow \tilde{R}(\tilde{C}^{\mathrm{Bayes}})$ for any choice of $k$ satisfying Stone's conditions.  Thus, by Corollary~\ref{cor:consistent}, if~\eqref{eq:symmetric} holds and $P_X(\tilde{\mathcal{S}} \setminus \mathcal{S}) = 0$, then the corrupted $k$nn classifier remains universally consistent.  In particular, in the special case of $\rho$-homogeneous noise, provided only that $\rho < 1/2$, this result tells us that the corrupted $k$nn classifier remains universally consistent. 

We now show that, under further regularity conditions on the data distribution $P$ and the noise mechanism, it is possible to give a more precise description of the asymptotic error properties of the corrupted $k$nn classifier.  Since our conditions on $P$, which are slight simplifications of those used in \citet{CBS:2017} to analyse the uncorrupted $k$nn classifier, are a little technical, we give an informal summary of them here, deferring formal statements of our assumptions A1--A4 to just before the proof of Theorem~\ref{thm:knnhet} in Section~\ref{sec:knnproofs}.  First, we assume that each of the class-conditional distributions has a density with respect to Lebesgue measure such that the marginal feature density $\bar{f}$ is continuous and positive.  It turns out that the dominant terms in the asymptotic expansion of the excess risk of $k$nn classifiers are driven by the behaviour of $P$ in a neighbourhood $\mathcal{S}^\epsilon$ of the set $\mathcal{S}$, which consists of points that are difficult to classify correctly, so we ask for further regularity conditions on the restriction of $P$ to $\mathcal{S}^\epsilon$.  In particular, we ask for both $\bar{f}$ and $\eta$ to have two well-behaved derivatives in $\mathcal{S}^\epsilon$, and for $\dot{\eta}$ to be bounded away from 0 on~$\mathcal{S}$.  This amounts to asking that the class-conditional densities, when weighted by the prior probabilities of each class, cut at an angle, and ensures that the set $\mathcal{S}$ is a $(d-1)$-dimensional orientable manifold.  Away from the set $\mathcal{S}^\epsilon$, we only require weaker conditions on $P_{X}$, and for $\eta$ to be bounded away from $1/2$.  Finally, we ask for two $\alpha$th moment conditions to hold, namely that $\int_{\mathbb{R}^d}  \|x\|^{\alpha} \, dP_{X}(x) < \infty$ and $\int_{\mathcal{S}}  \bar{f}(x_0)^{d/(\alpha+d)} \, d\mathrm{Vol}^{d-1}(x_0) < \infty$, where $d\mathrm{Vol}^{d-1}$ denotes the $(d-1)$-dimensional volume form on $\mathcal{S}$. 

For $\beta \in (0,1/2)$, let  $K_{\beta} = \{\lceil (n-1)^\beta \rceil, \ldots, \lfloor (n-1)^{1-\beta} \rfloor \}$ denote the set of values of $k$ to be considered for the $k$nn classifier.   Define
\[
B_{1} = \int_{\mathcal{S}} \frac{\bar{f}(x_0)}{4\|\dot\eta(x_0)\|} \, d\mathrm{Vol}^{d-1}(x_0),\quad B_{2} = \int_{\mathcal{S}} \frac{\bar{f}(x_0)^{1-4/d}}{\|\dot\eta(x_0)\|}a(x_0)^2 \, d\mathrm{Vol}^{d-1}(x_0),
\]
where
\[
a(x) = \frac{\sum_{j=1}^d \bigl\{\eta_j(x)\bar{f}_j(x) + \frac{1}{2}\eta_{jj}(x)\bar{f}(x)\bigr\}}{(d+2)a_d^{2/d}\bar{f}(x)}.
\]
We will also make use of a condition on the noise rates near the Bayes decision boundary:
	\begin{description}
\item[\normalfont{\emph{Assumption} B1.}] There exist $\delta > 0$ and a function $g:(1/2-\delta, 1/2 + \delta) \rightarrow [0,1)$ that is differentiable at $1/2$, with the property that for $x$ such that $\eta(x) \in (1/2-\delta, 1/2 + \delta)$, we have  $\rho_0(x) = g(\eta(x))$ and $\rho_1(x) = g(1 - \eta(x))$. 
	\end{description}
This assumption asks that, when $\eta(x)$ is close to $1/2$, the probability of label noise depends only on $x$ through $\eta(x)$, and moreover, this probability varies smoothly with $\eta(x)$.  In other words, Assumption~B1 says that the probability of mislabelling an observation with true class label 0 depends only on the extent to which it appeared to be from class 1; conversely, the probability of mislabelling an observation with true label 1 depends only, and in a symmetric way, on the extent to which it appeared to be from class 0.  To give just one of many possible examples, one could imagine that the probability that a doctor misdiagnoses a malignant tumour as benign depends on the extent to which it appears to be malignant, and vice versa.   We remark that \citet[Definition~11]{Menon:2016} introduce a related probabilistically transformed noise model, where $\rho_0 = g_0 \circ \eta$ and $\rho_1 = g_1 \circ \eta$, but they also require that $g_0$ and $g_1$ are increasing on $[0,1/2]$ and decreasing on $[1/2,1]$; see also \citet{Bylander:1997}.

\begin{theorem}
\label{thm:knnhet}
Assume \emph{A1}, \emph{A2}, \emph{A3} and \emph{A4}($\alpha$).   Suppose that $\rho_0$, $\rho_1$ are continuous, and that both
\[
\rho^* = \frac{1}{2} \sup_{x \in \mathbb{R}^d} \{\rho_0(x) + \rho_1(x)\} < \frac{1}{2}
\]
and 
\[
a^* = \sup_{x \in \mathcal{B}} \frac{\rho_1(x) - \rho_0(x)}{\{2\eta(x)-1\}\{1 - \rho_0(x) - \rho_1(x)\}} < 1.
\]
Moreover, assume \emph{B1} holds with the additional requirement that $g$ is twice continuously differentiable, $\dot{g}(1/2) > 2g(1/2) -1$ and that $\ddot{g}$ is uniformly continuous. Then we have two cases:

(i) Suppose that $d \geq 5$ and $\alpha > 4d/(d-4)$.  Then for each $\beta \in (0,1/2)$,
\[
R(\tilde{C}^{k\mathrm{nn}}) - R(C^{\mathrm{Bayes}}) = \frac{B_1}{k\{1- 2g(1/2) + \dot{g}(1/2)\}^{2}}  + B_{2} \Bigl(\frac{k}{n}\Bigr)^{4/d} + o\biggl(\frac{1}{k} + \Bigl(\frac{k}{n}\Bigr)^{4/d}\biggr)
\]
as $n \to \infty$, uniformly for $k \in K_{\beta}$.

 (ii) Suppose that either $d \leq 4$, or, $d \geq 5$ and $\alpha \leq 4d/(d-4)$.  Then for each $\beta \in (0,1/2)$ and each $\epsilon> 0$ we have
\[
R(\tilde{C}^{k\mathrm{nn}}) - R(C^{\mathrm{Bayes}}) =   \frac{B_1}{k\{1- 2g(1/2) + \dot{g}(1/2)\}^{2}}  + o\Bigl(\frac{1}{k} +\Bigl(\frac{k}{n}\Bigr)^{\frac{\alpha}{\alpha+d} - \epsilon}\Bigr)
\]
as $n \to \infty$, uniformly for $k \in K_{\beta}$.  
\end{theorem}
The proof of Theorem~\ref{thm:knnhet} is given in Section~\ref{sec:knnproofs}, and involves two key ideas.  First, we demonstrate that the conditions assumed for $\eta$ also hold for the corrupted regression function~$\tilde{\eta}$.  Second, we show that the dominant asymptotic contribution to the desired excess risk $R(\tilde{C}^{k\mathrm{nn}}) - R(C^{\mathrm{Bayes}})$ is $\{\tilde{R}(\tilde{C}^{k\mathrm{nn}}) - \tilde{R}(\tilde{C}^{\mathrm{Bayes}})\}/\{1-2g(1/2) + \dot{g}(1/2)\}$, a scalar multiple of the excess risk when predicting a noisy label.  We then conclude the argument by appealing to \citet[Theorem~1]{CBS:2017}, and of course, can recover the conclusion of that result for noiseless labels as a special case of Theorem~\ref{thm:knnhet} by setting $g = 0$. 

In the conclusion of Theorem~\ref{thm:knnhet}(i), the terms $B_1/[k\{1- 2g(1/2) + \dot{g}(1/2)\}^{2}]$ and $B_{2}(k/n)^{4/d}$ can be thought of as the leading order contributions to the variance and squared bias of the corrupted $k$nn classifier respectively.  It is both surprising and interesting to note that the type of label noise considered here affects only the leading order variance term compared with the noiseless case; the dominant bias term is unchanged.  To give a concrete example, $\rho$-homogeneous noise satisfies the conditions of Theorem~\ref{thm:knnhet}, and in the setting of Theorem~\ref{thm:knnhet}(i), we see that the dominant variance term is inflated by a factor of $(1-2\rho)^{-2}$.

We now quantify the relative asymptotic performance of the corrupted $k$nn and uncorrupted $k$nn classifiers.  Since this performance depends on the choice of~$k$ in each case, we couple these choices together in the following way: given any $k$ to be used by the uncorrupted classifier $C^{k\mathrm{nn}}$, and given the function $g$ from Theorem~\ref{thm:knnhet}, we consider the choice
\begin{equation}
\label{Eq:kg}
k_g =  \bigl\lfloor \{1-2g(1/2) + \dot{g}(1/2)\}^{-2d/(d+4)} k \bigr\rfloor
\end{equation}
for the noisy label classifier $\tilde{C}^{k\mathrm{nn}}$.  This coupling reflects the ratio of the optimal choices of~$k$ for the corrupted and uncorrupted label settings.
\begin{corollary} 
\label{cor:knnRR}
Under the assumptions of Theorem~\ref{thm:knnhet}(i), and provided that $B_2 > 0$, we have that for any $\beta \in (0,1/2)$,
\begin{equation} 
\label{eq:RR}
\frac{R(\tilde{C}^{k_g\mathrm{nn}}) - R(C^{\mathrm{Bayes}})}{R(C^{k\mathrm{nn}}) - R(C^{\mathrm{Bayes}})}  
\rightarrow \{1 - 2g(1/2) + \dot{g}(1/2)\}^{-8/(d+4)},
\end{equation} 
as $n \rightarrow \infty$, uniformly for $k \in K_\beta$.
\end{corollary} 
If $\dot{g}(1/2) > 2 g(1/2)$, then the limiting regret ratio in \eqref{eq:RR} is less than 1 -- this means that the label noise helps in terms of the asymptotic performance!  This is due to the fact that, under the noise model in Theorem~\ref{thm:knnhet}, if $\dot{g}(1/2) > 2 g(1/2)$ then for points $X_i$ with $\eta(X_i)$ close to $1/2$, the noisy labels $\tilde{Y}_{i}$ are more likely than the true labels $Y_i$ to be equal to the Bayes labels, $\mathbbm{1}_{\{\eta(X_{i}) \geq 1/2\}}$.  To understand this phenomenon, first note that by rearranging~\eqref{eq:tildeeta}, we have
\begin{align*}
  \tilde{\eta}(x) - 1/2 = \{\eta(x) - 1/2\}\{1-\rho_0(x) - \rho_1(x)\} + \frac{1}{2}\{\rho_0(x) - \rho_1(x)\}.
\end{align*}
Thus $\tilde \eta(x)-1/2 = \eta(x)-1/2$ for $x \in \mathcal{S}$ using B1.  On the other hand, for $x\in \mathcal{S}^c$,  we have
\begin{equation}\label{eq: 001}
\tilde {\eta}(x) - 1/2 = \{\eta(x) - 1/2\} \biggl(1-\rho_0(x) - \rho_1(x) + \frac{\rho_0(x) - \rho_1(x)}{2\eta(x)-1}\biggr). 
\end{equation}
We next study the term in the second parentheses on the right-hand side above.  Write $t = \eta(x) - 1/2$. Then, for $x$ such that $|\eta(x) -1/2|  \in (0,\delta)$, we have  $\rho_0(x) = g(1/2+t)$ and $\rho_1(x) = g(1/2-t)$. It follows, for such $x$, that
\begin{align*}
 1 \! - \! \rho_0(x)\! -\! \rho_1(x) \! + \!\frac{\rho_0(x) - \rho_1(x)}{2\eta(x)-1} &= 1-g(1/2+t) - g(1/2-t) + \frac{g(1/2+t)-g(1/2-t)}{2t} \\
  &\rightarrow 1 - 2g(1/2) + \dot{g}(1/2)
\end{align*}
as $|t|\searrow 0$. Since $1-2g(1/2)+\dot{g}(1/2)>1$, we obtain that for any $\varepsilon \in \bigl(0, \dot{g}(1/2)/2-g(1/2)\bigr)$, there exists $\delta_0 \in (0, \delta)$ such that for all $x$ with $|\eta(x) - 1/2|\in (0,\delta_0)$, we have that 
\[
1-\rho_0(x) - \rho_1(x) + \frac{\rho_0(x) - \rho_1(x)}{2(\eta(x)-1/2)} > 1 -2g(1/2) + \dot{g}(1/2)-\varepsilon>1. 
\]
This together with \eqref{eq: 001} ensures that, for all $x$ such that $|\eta(x)-1/2| \in (0,\delta_0)$, we have 
\[
|\tilde\eta(x)-1/2| > |\eta(x)-1/2|.
\]
 
 \begin{example} 
 	\emph{Suppose that for some $g_0 \in (0,1/2)$ and $h_0 > 2 - 1/g_0$ we have $g(1/2 + t) = g_0(1 + h_0 t)$ for $t \in (-\delta,\delta)$.  Then $g(1/2) = g_0$ and $\dot{g}(1/2) = g_0 h_0$,  which gives $1 - 2g(1/2) + g'(1/2) = 1 + (h_0 - 2) g_0 $.   We therefore see from Corollary~\ref{cor:knnRR} that if $h_0 < 2$, then the limiting regret ratio is greater than $1$, but if $h_0 > 2$, then the limiting regret ratio is less than one, so the label noise aids performance. }
 \end{example} 

\subsection{Support vector machine classifiers} 
\label{sec:SVM}
In general, the term \textit{support vector machines} (SVM) refers to classifiers of the form
\[
C^{\mathrm{SVM}}(x) = C_n^{\mathrm{SVM}}(x) = \left\{ \begin{array}{ll} 1& \text{\quad if } \hat{f} (x) \geq 0
\\ 0 & \text{\quad otherwise,} \end{array} \right.
\]
where the \textit{decision function} $\hat{f}$ satisfies 
\begin{equation*}
\label{eq:SVMopt}
\hat{f} \in \argmin_{f \in H} \biggl\{\frac{1}{n} \sum_{i = 1}^{n} L(Y_{i}, f(X_{i})) + \Omega(\lambda, \|f\|_{H})\biggr\}.
\end{equation*} 
See, for example, \citet{Cortes:1995} and \citet{Steinwart:2008}.  Here $L: \mathbb{R} \times \mathbb{R} \rightarrow \mathbb{R}$ is a loss function, $\Omega : \mathbb{R} \times \mathbb{R} \rightarrow \mathbb{R}$ is a regularization function,  $
\lambda > 0$ is a tuning parameter and $H$ is a \textit{reproducing kernel Hilbert space} (RKHS) with norm $\|\cdot\|_{H}$ \citep[Chapter~4]{Steinwart:2008}. 

We focus throughout on the \emph{L1-SVM}, where $L(y, t) = \max\{0, 1 - (2y-1)t\}$ is the \textit{hinge loss} function and $\Omega(\lambda, t) = \lambda t^{2}$.  Let $K : \mathbb{R}^d \times \mathbb{R}^d \rightarrow \mathbb{R}$ be the positive definite kernel function associated with the RKHS. We consider the Gaussian radial basis function, namely $K(x,x') = \exp(-\sigma^2 \|x - x'\|^{2} )$, for $\sigma > 0$.  The corrupted SVM classifier is
\begin{equation}
\label{eq:SVMK}
\tilde{C}^{\mathrm{SVM}}(x) = \tilde{C}_n^{\mathrm{SVM}}(x) = \left\{ \begin{array}{ll} 1& \text{\quad if } \tilde{f}(x) \geq 0 
\\ 0 & \text{\quad otherwise,} \end{array} \right.
\end{equation}
where
\begin{equation}
\label{eq:SVMopttilde}
\tilde{f} \in \argmin_{f \in H} \biggl\{\frac{1}{n} \sum_{i = 1}^{n}  \max\{0, 1 - (2\tilde{Y}_i-1)f(X_{i})\} + \lambda \|f\|^2_{H}\biggr\}.
\end{equation}

\citet[][Corollary~3.6 and Example~3.8]{Steinwart:2005} show that the uncorrupted L1-SVM classifier is consistent as long as $P_X$ is compactly supported and $\lambda = \lambda_{n}$ is such that $\lambda_{n} \rightarrow 0$ but $n\lambda_{n}/(|\log{\lambda_{n}} |^{d+1}) \rightarrow \infty$.  Therefore, under these conditions, provided that~\eqref{eq:symmetric} holds and $P_X(\tilde{\mathcal{S}} \setminus \mathcal{S}) = 0$, by Corollary~\ref{cor:consistent}, we have that  $R(\tilde{C}^{\mathrm{SVM}}) \rightarrow R(C^{\mathrm{Bayes}})$ as $n \rightarrow \infty$. 

Under further conditions on the noise probabilities and the distribution $P$, we can also provide more precise control of the excess risk for the SVM classifier.  Our analysis will make use of the results in \citet{Steinwart:2007}, who study the rate of convergence of the SVM classifier with Gaussian kernels in the noiseless label setting.  Other works of note on the rate of convergence of SVM classifiers include \citet{Lin:1999} and  \citet{Blanchard:2008}; see also \citet[Chapters~6 and~8]{Steinwart:2008}.  

We recall two definitions used in the perfect labels context. The first of these is the well-known \textit{margin assumption} of, for example, \citet{Audibert:2007}.  We say that the distribution $P$ satisfies the margin assumption with parameter $\gamma_1 \in [0,\infty)$ if there exists $\kappa_{1}>0$ such that
\[
P_{X}(\{x \in \mathbb{R}^{d} : 0 < |\eta(x) - 1/2| \leq t\}) \leq \kappa_{1} t^{\gamma_1}
\]
for all $t>0.$  If $P$ satisfies the margin assumption for all $\gamma_1 \in [0,\infty)$ then we say $P$ satisfies the margin assumption with parameter $\infty$.  The margin assumption controls the probability mass of the region where $\eta$ is close to $1/2$.

The second definition we need is that of the \textit{geometric noise exponent} \citep[Definition~2.3]{Steinwart:2007}.  Let $\mathcal{S}_{+} = \{x\in \mathbb{R}^{d} : \eta(x) > 1/2\}$ and $\mathcal{S}_{-} = \{x\in \mathbb{R}^{d} : \eta(x) < 1/2\}$, and for $x \in \mathbb{R}^{d}$, let $\tau_{x} = \inf_{x' \in \mathcal{S} \cup \mathcal{S}_{+}} \|x - x'\| + \inf_{x' \in \mathcal{S} \cup \mathcal{S}_{-}} \|x - x'\|$.  We say that the distribution $P$ has geometric noise exponent $\gamma_2 \in [0,\infty)$ if there exists $\kappa_2>0$, such that
\[
\int_{\mathbb{R}^{d}} |2\eta(x) - 1| \exp\Bigl(-\frac{\tau_{x}^{2}}{t^2}\Bigr) \, dP_{X}(x) \leq \kappa_{2}t^{\gamma_2 d}
\]
for all $t>0.$  If $P$ has geometric noise exponent $\gamma_2$ for all $\gamma_2 \in [0,\infty)$ then we say it has geometric noise exponent $\infty$.

Under these two conditions, \citet[Theorem 2.8]{Steinwart:2007} show that, if $P_X$ is supported on the closed unit ball, then for appropriate choices of the tuning parameters, the SVM classifier 
achieves a convergence rate of  $O(n^{-\Gamma + \epsilon})$ for every $\epsilon >0$, where 
\[
\Gamma = \left\{ \begin{array}{ll} 
\frac{\gamma_2}{2\gamma_2 + 1} & \text{\quad if } \gamma_2 \leq \frac{\gamma_1 + 2}{2\gamma_1} \\
\frac{2\gamma_2(\gamma_1+1)}{2\gamma_2(\gamma_1 + 2) + 3 \gamma_1 + 4} & \text{\quad otherwise.} \end{array} \right.
\]

In the imperfect labels setting, and under our stronger assumption on the noise mechanism when $\eta$ is close to $1/2$, we see that the SVM classifier trained with imperfect labels satisfies the same bound on the rate of convergence as in the perfect labels case. 
\begin{theorem}
\label{thm:SVMhet}
Suppose that $P$ satisfies the margin assumption with parameter $\gamma_1 \in [0, \infty]$, has geometric noise exponent $\gamma_2 \in (0,\infty)$ and that $P_X$ is supported on the closed unit ball.  Assume the conditions of Theorem~\ref{thm:hetnoise}(ii) and B1 holds.  Then
\[
R(\tilde{C}^{\mathrm{SVM}}) - R(C^{\mathrm{Bayes}}) = O(n^{-\Gamma+\epsilon}),
\]
as $n \rightarrow \infty$, for every $\epsilon > 0$.  If $\gamma_2 = \infty$, then the same conclusion holds provided $\sigma_n = \sigma$ is a constant with $\sigma > 2d^{1/2}$. 
\end{theorem} 

\subsection{Linear discriminant analysis}
\label{sec:LDA}
If $P_0 = N_d(\mu_0, \Sigma)$ and $P_1 = N_d(\mu_1, \Sigma)$, then the Bayes classifier is
\begin{equation}
\label{eq:LDABayes}
C^{\mathrm{Bayes}}(x) = \left\{ \begin{array}{ll} 1  & \mbox{\quad if $\log\bigl(\frac{\pi_1}{\pi_0}\bigr) +  \bigl(x - \frac{{\mu}_1 + {\mu}_0}{2}\bigr)^T {\Sigma}^{-1} ({\mu}_1 - {\mu}_0) \geq 0$}
\\ 0 & \mbox{\quad otherwise.} \end{array} \right.
\end{equation}
The Bayes risk can be expressed in terms of $\pi_0, \pi_1$, and the squared Mahalanobis distance $\Delta^2 = (\mu_1 - \mu_0)^T \Sigma^{-1} (\mu_1 - \mu_0)$ between the classes as 
\[
R(C^{\mathrm{Bayes}}) = \pi_0\Phi\biggl(\frac{1}{\Delta} \log\Bigl(\frac{\pi_1}{\pi_0}\Bigr) - \frac{\Delta}{2}\biggr) +  \pi_1\Phi\biggl(\frac{1}{\Delta}\log\Bigl(\frac{\pi_0}{\pi_1}\Bigr) - \frac{\Delta}{2}\biggr),
\]
where $\Phi$ denotes the standard normal distribution function. 

The LDA classifier is constructed by substituting training data estimates of $\pi_0, \pi_1, \mu_0, \mu_1,$ and $\Sigma$ in to \eqref{eq:LDABayes}.   With imperfect training data labels, and for $r = 0,1$, we define estimates $\hat{\pi}_r = n^{-1}\sum_{i=1}^n \mathbbm{1}_{\{\tilde{Y}_i = r\}}$ of $\pi_r$, as well as estimates $\hat{\mu}_r = \sum_{i = 1}^n X_i \mathbbm{1}_{\{\tilde{Y}_i = r\}}/\sum_{i=1} ^n \mathbbm{1}_{\{\tilde{Y}_i = r\}}$ of the class-conditional means $\mu_r$, and set
\[
 \hat{\Sigma} = \frac{1}{n-2} \sum_{i = 1}^{n} \sum_{r = 0}^1 (X_i -  \hat{\mu}_r) ( X_i -  \hat{\mu}_r)^T \mathbbm{1}_{\{\tilde{Y}_i = r\}}.
\] 
This allows us to define the corrupted LDA classifier
\[
\tilde{C}^{\mathrm{LDA}}(x) = \tilde{C}_n^{\mathrm{LDA}}(x) = \left\{ \begin{array}{ll} 1  & \text{\quad if } \log\bigl(\frac{\hat{\pi}_1}{\hat{\pi}_0}\bigr) +  \bigl(x - \frac{\hat{\mu}_1 + \hat{\mu}_0}{2}\bigr)^T \hat{\Sigma}^{-1} (\hat{\mu}_1 - \hat{\mu}_0) \geq 0 
\\ 0 & \text{\quad otherwise.} \end{array} \right.
\]
Consider now the $\rho$-homogeneous noise setting.  In this case, writing $\tilde{P}_{r}$, $r\in \{0,1\}$, for the distribution of $X \mid \{\tilde{Y} = r\}$, we have $\tilde{P}_r =  p_r N_d(\mu_r, \Sigma) +  (1-p_r) N_d(\mu_{1-r}, \Sigma)$, where $p_r = \pi_r  (1-\rho)/\{\pi_r (1-\rho) + \pi_{1-r}\rho\}$. Notice that while $\hat{\pi}_r, \hat{\mu}_r$ and $\hat{\Sigma}$ are intended to be estimators of $\pi_r, \mu_r$ and $\Sigma$, respectively, with label noise these will in fact be consistent estimators of $\tilde{\pi}_r =  \pi_r (1-\rho) + \pi_{1-r} \rho$, $\tilde{\mu}_r = p_r \mu_r +  (1-p_r) \mu_{1-r}$, and $\tilde{\Sigma} = \Sigma + \alpha(\mu_1-\mu_0)(\mu_1-\mu_0)^T$, respectively, where $\alpha > 0$ is given in the proof of Theorem~\ref{thm:LDAhomo}.

We will also make use of the following well-known lemma in the homogeneous label noise case \citep[e.g.][Theorem~1]{Ghosh:2015}, which holds for an arbitrary classifier and data generating distribution.  We include the short proof for completeness.
\begin{lemma}
\label{lem:tilde}
For $\rho$-homogeneous noise with $\rho\in [0,1/2)$ and for any classifier $C$, we have $R(C) = \{\tilde{R}(C) - \rho\}/(1- 2\rho)$.  Moreover, $R(C)  - R(C^{\mathrm{Bayes}})  =  \bigl\{\tilde{R}(C)  - \tilde{R}(C^{\mathrm{Bayes}})\bigr\}/(1- 2\rho)$.
\end{lemma}

The following is the main result of this subsection. 
\begin{theorem} 
\label{thm:LDAhomo}
Suppose that $P_r =  N_d(\mu_r, \Sigma)$ for $r = 0,1$ and that the noise is $\rho$-homogeneous with $\rho \in [0,1/2)$.  Then
\[
  \lim_{n \rightarrow \infty} \tilde{C}^{\mathrm{LDA}}(x) = \left\{ \begin{array}{ll} 1  & \mbox{\quad if $c_0 +  \bigl(x - \frac{\mu_1 + \mu_0}{2}\bigr)^T \Sigma^{-1} (\mu_1 - \mu_0) > 0$}
\\ 0 & \mbox{\quad if $c_0 + \bigl(x - \frac{\mu_1 + \mu_0}{2}\bigr)^T \Sigma^{-1} (\mu_1 - \mu_0) < 0$,} \end{array} \right.
\]
where
\[
c_0 \! = \! \Bigl\{(1-2\rho) + \frac{\rho(1\!-\!\rho)(1 \! + \! \pi_0\pi_1\Delta^2)}{(1-2\rho)\pi_{1}\pi_{0}}   \Bigr\} \log\Bigl(\frac{(1-2\rho)\pi_1 + \rho}{(1-2\rho)\pi_0 + \rho} \Bigr)   - \frac{(\pi_{1} - \pi_{0})\rho(1-\rho)\Delta^{2}}{2\{(1\!-\!2\rho)^{2} \pi_1 \pi_{0} + \rho(1\!-\!\rho)\}} .
\]
As a consequence,
\begin{equation}
  \label{Eq:LimitingLDARisk}
\lim_{n \rightarrow \infty} R(\tilde{C}^{\mathrm{LDA}}) = \pi_0\Phi\biggl(\frac{c_0}{\Delta} - \frac{\Delta}{2}\biggr) +  \pi_1\Phi\biggl(-\frac{c_0}{\Delta} - \frac{\Delta}{2}\biggr) \geq R(C^{\mathrm{Bayes}}).
\end{equation}
For each $\rho \in (0,1/2)$ and $\pi_0 \neq \pi_1$, there exists a unique value of $\Delta > 0$ for which equality in the inequality in~\eqref{Eq:LimitingLDARisk} is attained.
\end{theorem}
The first conclusion of this theorem reveals the interesting fact that, regardless of the level $\rho \in (0,1/2)$ of label noise, the limiting corrupted LDA classifier has a decision hyperplane that is parallel to that of the Bayes classifier; see also \citet{Lachenbruch:1966} and \citet[Corollary~1]{Manwani:2013}.  However, for each fixed $\rho \in (0,1/2)$ and $\pi_0 \neq \pi_1$, there is only one value of $\Delta > 0$ for which the offset is correct and the corrupted LDA classifier is consistent.  

\section{Numerical comparison} 
\label{sec:sims} 
In this section, we investigate empirically how the different types of label noise affect the performance of the $k$-nearest neighbour, support vector machine and linear discriminant analysis classifiers.  We consider two different model settings for the pair $(X,Y)$: 

Model~1: Let $\mathbb{P}(Y=1) = \pi_1 \in \{$0.5, 0.9$\}$ and  $X \mid \{Y = r\} \sim N_{d}(\mu_{r}, I_d)$, where $\mu_{1} = (3/2, 0, \ldots, 0)^T = -\mu_{0} \in \mathbb{R}^{d}$ and $I_d$ denotes the $d$ by $d$ identity matrix.  

Model~2: For $d \geq 2$, let $X \sim U([0,1]^{d})$ and $\mathbb{P}(Y =1 \mid X = x) = \eta(x_1, \ldots, x_d) = \min\{ 4(x_1 - 1/2)^2 +  4(x_2 - 1/2)^2, 1\}$.

In each setting, our risk estimates are based on an uncorrupted test set of size 1000, and we repeat each experiment 1000 times. This ensures that all standard errors are less than 0.4$\%$ and 0.14 for the risk and regret ratio estimates, respectively; in fact, they are often much smaller.
\begin{figure}[ht!]
	\centering
\includegraphics[width=\textwidth]{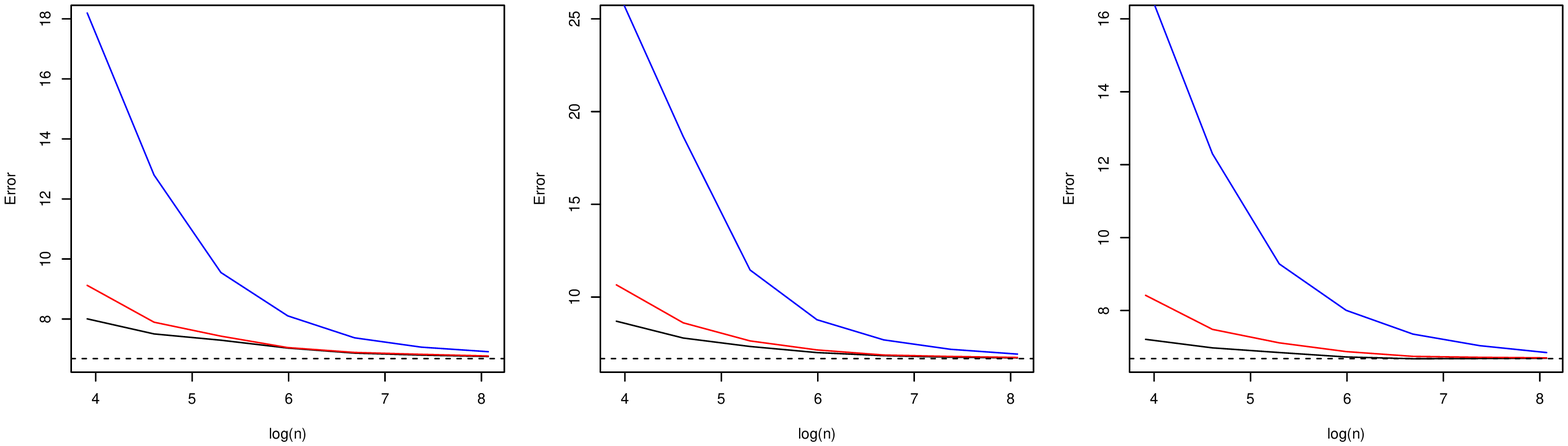}
\includegraphics[width=\textwidth]{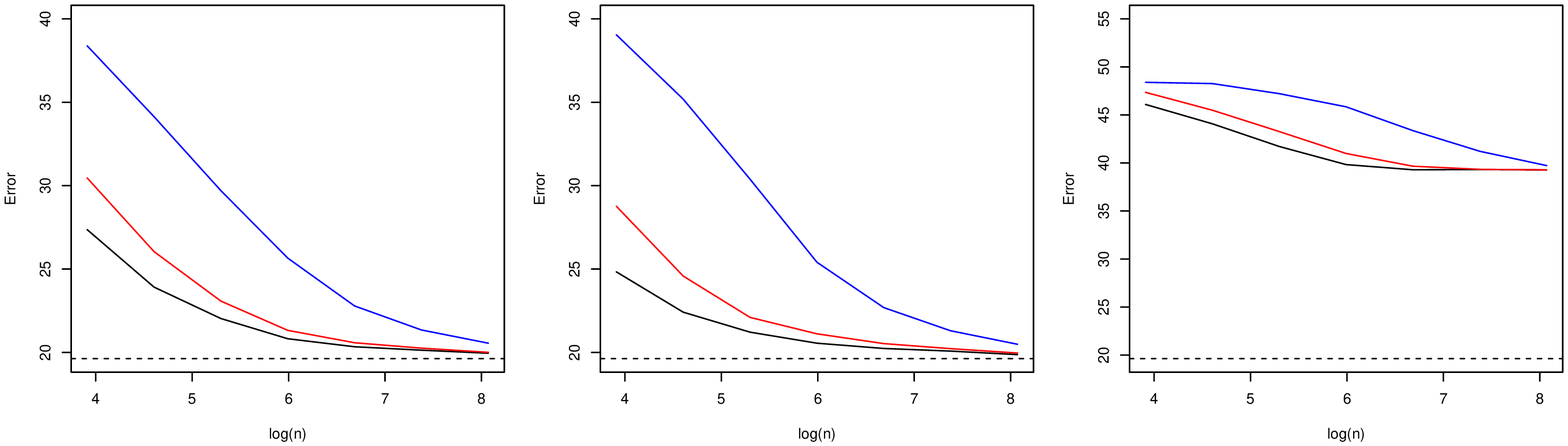}
\caption{{\small Risk estimates for the $k$nn (left), SVM (middle) and LDA (right) classifiers. Top: Model~1, $d = 2$, $\pi_{1} =$ 0.5, Bayes risk = 6.68\%, shown as the black dotted line.  Bottom: Model~2, $d = 2$, Bayes risk = 19.63\%.   We present the results without label noise (black) and with homogeneous label noise at rate $\rho =$ 0.1 (red) and 0.3 (blue). }} 
\label{fig:errorplot1a}
\end{figure}
\begin{figure}[ht!]
\centering
\includegraphics[width=0.61\textwidth]{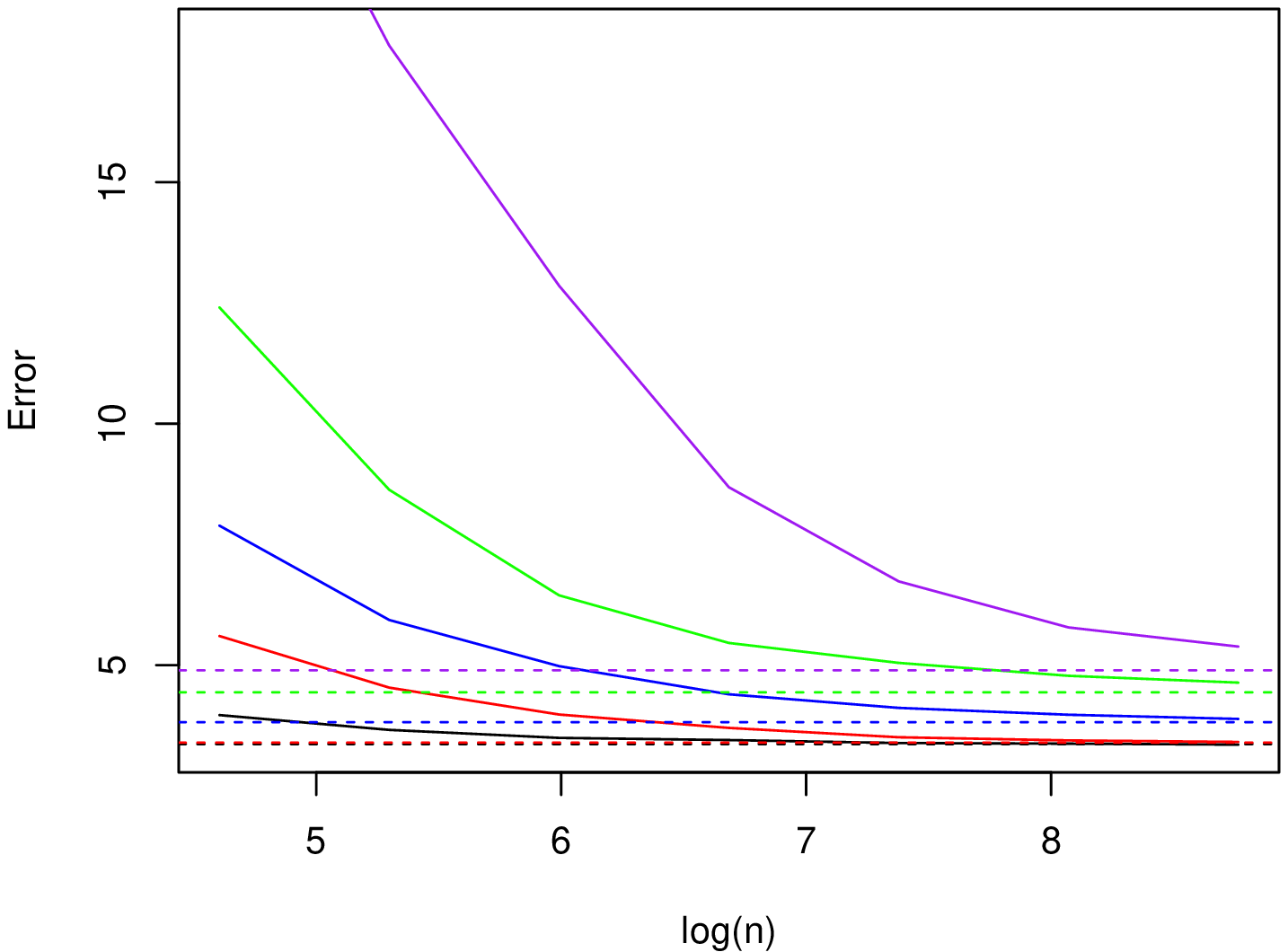}
\caption{{\small Risk estimates for the LDA classifier. Model~1, $d=5$, $\pi_1 =$ 0.9, Bayes risk = 3.37\%.  We present the estimated error without label noise (black) and with homogeneous label noise at rate $\rho =$ 0.1 (red), 0.2 (blue), 0.3 (green) and 0.4 (purple).  The dotted lines show the corresponding asymptotic limit as given by Theorem~\ref{thm:LDAhomo}.}}
\label{fig:errorplotLDA}
\end{figure}

Our first goal is to illustrate numerically our consistency and inconsistency results for the $k$nn, SVM and LDA classifiers.  In Figure~\ref{fig:errorplot1a} we present estimates of the risk for the three classifiers with different levels of homogeneous label noise.   We see that for Model~1 when the class prior probabilities are equal, all three classifiers perform well and in particular appear to be consistent, even when as many as 30\% of the training data labels are incorrect on average.   For the $k$nn and SVM classifiers we observe very similar results for Model~2; the LDA classifier does not perform well in this setting, however, since the Bayes decision boundary is non-linear.  These conclusions are in accordance with Corollary~\ref{cor:consistent} and Theorem~\ref{thm:LDAhomo}.

We further investigate the effect of homogeneous label noise on the performance of the LDA classifier for data from Model~1, but now when $d=5$ and the class prior probabilities are unbalanced.   Recall that in Theorem~\ref{thm:LDAhomo} we derived the asymptotic limit of the risk in terms of the Mahalanobis distance between the true class distributions, the class prior probabilities and the noise rate.   In Figure~\ref{fig:errorplotLDA}, we present the estimated risks of the LDA classifier for data from Model~1 with $\pi_1 =$ 0.9 for different homogeneous noise rates alongside the limit as specified by Theorem~\ref{thm:LDAhomo}.  This articulates the inconsistency of the corrupted LDA classifier, as observed in Theorem~\ref{thm:LDAhomo}.

Finally, we study empirically the asymptotic regret ratios for the $k$nn and SVM classifiers. We focus on the noise model in Example 2 in Section~\ref{sec:asymptotic}, where the label errors occur at random as follows: fix $g_0 \in (0,1/2)$, $h_0 > 2 - 1/g_0$, we let $g(1/2 + t) = \max[0, \min\{g_0(1 + h_0 t), 2g_{0}\}]$, then set $\rho_{0}(x) = g(\eta(x))$ and $\rho_{1}(x) = g(1-\eta(x))$.   In particular, we use the following settings: (i) $g_{0} = $ 0.1, $h_{0} = 0$; (ii) $g_{0} = $ 0.1, $h_{0} = -1$; (iii) $g_{0} = $ 0.1, $h_{0} = 1$; (iv) $g_{0} = $ 0.1, $h_{0} = 2$; (v) $g_{0} = $ 0.1, $h_{0} = 3$.  Noise setting (i), where $h_0 = 0$, corresponds to $g_0$-homogeneous noise.   

\begin{figure}[ht!]
	\centering
\includegraphics[width=0.9\textwidth]{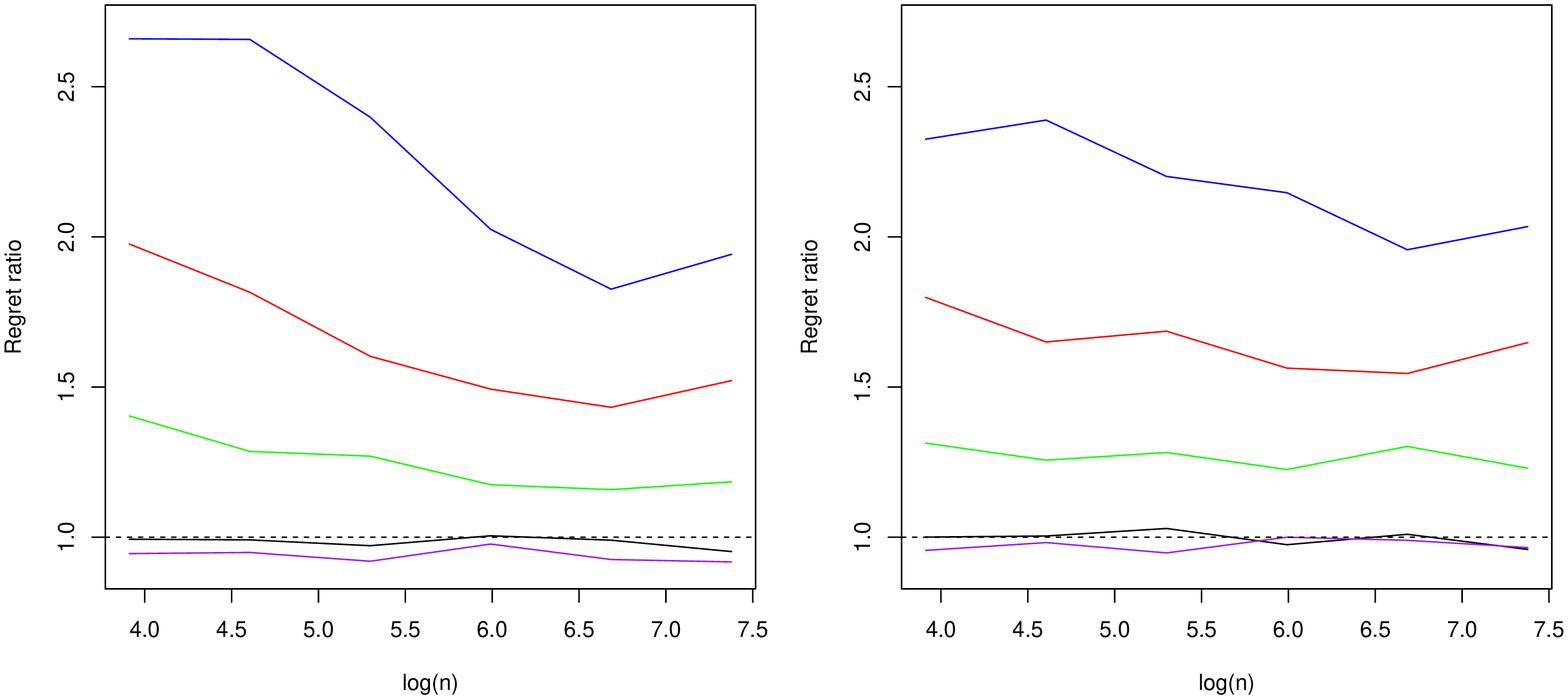}
\includegraphics[width=0.9\textwidth]{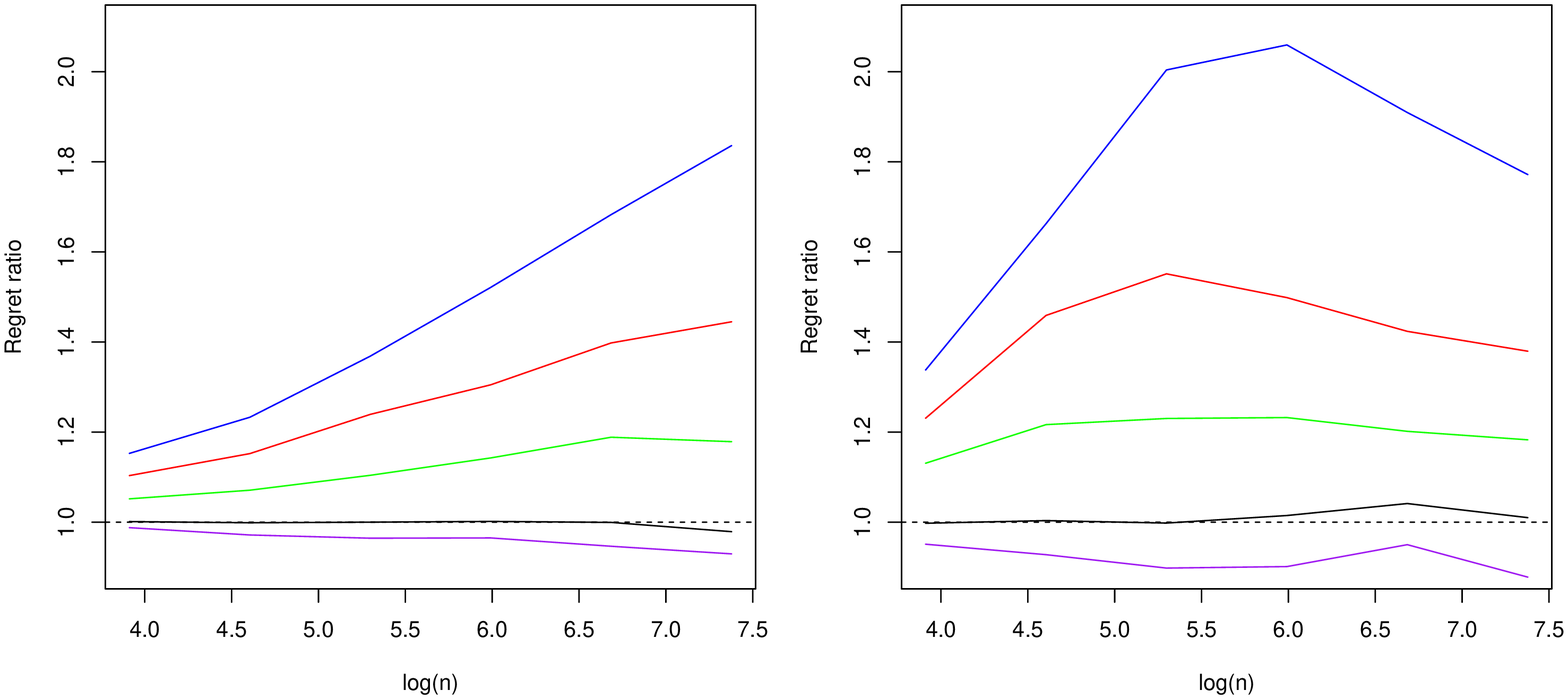}
\caption{{\small Estimated regret ratios for the $k$nn (left) and SVM (right) classifiers. Top: Model~1, with $d = 5$ and $\pi_{1} = $ 0.5. Bottom: Model~2, with $d = 5$.  We present the results with label noise of type (i -- red), (ii -- blue), (iii -- green), (iv -- black), and  (v -- purple).}} 
\label{fig:errorplot4a}
\end{figure}

For the $k$nn classifier, where $k$ is chosen to satisfy the conditions of Corollary~\ref{cor:knnRR}, our theory says that when $d=5$ in Models~1 and~2, the asymptotic regret ratios in the five noise settings are 1.22, 1.37, 1.10, 1 and 0.92 respectively.  We see from the left-hand plots of Figure~\ref{fig:errorplot4a} that, for $k$ chosen separately in the corrupted and uncorrupted cases via cross-validation, the empirical results provide good agreement with our theory, especially in the last three settings.  Reasons for the slight discrepancies between our asymptotic theory and empirically observed regret ratios in the first two noise settings include the following facts: the choices of $k$ in the noisy and noiseless label settings do not necessarily satisfy~\eqref{Eq:kg} exactly; the asymptotics in $n$ may not have fully `kicked in'; and Monte Carlo error (when $n$ is large, we are computing the ratio of two small quantities, so the standard error tends to be larger).  The performance of the SVM classifier is similar to that of the $k$nn classifier for both models.  

Finally, we discuss tuning parameter selection.  We have seen that for the $k$nn classifier the the choice of $k$ is important for achieving the optimal bias--variance trade-off; see also \citet{Hall:2008}.  Similarly, we need to choose an appropriate value of $\lambda$ for the SVM classifier; in practice, this is typically done via cross-validation.  When the classifier $\tilde{C}$ is trained with $\rho$-homogeneous noisy labels, we would like to select a tuning parameter to minimize $R(\tilde{C})$, but since the training data is corrupted, a tuning parameter selection method will target the minimizer of $\tilde{R}(\tilde{C})$.  However, by Lemma~\ref{lem:tilde}, we have that $R(\tilde{C}) = \{\tilde{R}(\tilde{C}) - \rho\}/(1- 2\rho)$, and it follows that our tuning parameter selection method requires no modification when trained with noisy labels.  In the heterogeneous noise case, however, we do not have this direct relationship; see \citet{Inouye:2017} for more on this topic.

In our simulations, we chose $k$ for the $k$nn classifier and $\lambda$ for the SVM classifier via leave-one-out and 10-fold cross-validation respectively, where the cross-validation was performed over the noisy training dataset. Moreover, for the SVM classifier, we used the default choice $\sigma^2= 1/d$ for the hyper-parameter for the kernel function.

\appendix

\section{Proofs and an illustrative example} 
\label{sec:proofs}

\subsection{Proofs from Section~\ref{sec:finite}}

\begin{proof}[Proof of Theorem~\ref{thm:hetnoise}]
	(i) First, we have that for $P_X$-almost all $x \in \mathcal{B}$, 
	\begin{align}
	\label{eq:etaetatilde}
	\tilde{\eta}(x) - 1/2 & =  \{\eta(x) - 1/2\} \{1 - \rho_0(x) - \rho_1(x)\} + \frac{1}{2} \{\rho_0(x) - \rho_1(x)\} \nonumber
	\\& = \{\eta(x) - 1/2\} \{1 - \rho_0(x) - \rho_1(x)\}\biggl( 1 -   \frac{\rho_1(x) - \rho_0(x)}{\{2\eta(x)-1\}\{1 - \rho_0(x) - \rho_1(x)\}} \biggr). 
	\end{align} 
	Thus, for $P_X$-almost all $x \in \mathcal{B}$,  we have $\{\rho_1(x) - \rho_0(x)\}/[\{2\eta(x)-1\}\{1 - \rho_0(x) - \rho_1(x)\}] < 1$ if and only if
	\[
	\mathrm{sgn}\{\tilde{\eta}(x) - 1/2\} = \mathrm{sgn}\{\eta(x) - 1/2\}.
	\]
	This completes the proof of~\eqref{eq:noisecond}. It follows that, if $P_X(\mathcal{A}^c \cap \mathcal{S}^c)=0$, then $P_X(\{x \in \mathcal{B} : \tilde{C}^{\mathrm{Bayes}}(x)= C^{\mathrm{Bayes}}(x)\}^c \cap \mathcal{S}^c)=0$.  In other words $P_X(\{x \in \mathcal{S}^c: \tilde{C}^{\mathrm{Bayes}}(x)\neq C^{\mathrm{Bayes}}(x) \})=0$, i.e.~(2) holds.  Here we have used the fact that $\mathcal{A} \subseteq \mathcal{B}$, so if $P_X(\mathcal{A}^c \cap \mathcal{S}^c)=0$, then $P_X(\mathcal{B}^c \cap \mathcal{S}^c)=0$.
	
	(ii) For the proof of this part, we apply Proposition~\ref{prop:prelim}.  First, since~\eqref{eq:symmetric} holds, we have $\tilde{R}(C^{\mathrm{Bayes}}) = \tilde{R}(\tilde{C}^{\mathrm{Bayes}})$.   From~\eqref{eq:etaetatilde}, we have that for $P_X$-almost all $x \in \mathcal{B}$,
	\begin{align}
	\label{Eq:Akappa}
	|2\tilde{\eta}(x) - 1| &= |2\eta(x) - 1| \{1 - \rho_0(x) - \rho_1(x)\}\Bigl( 1 -   \frac{\rho_1(x) - \rho_0(x)}{\{2\eta(x)-1\}\{1 - \rho_0(x) - \rho_1(x)\}} \Bigr) \nonumber
	\\ & \geq  |2\eta(x) - 1| (1 - 2\rho^*) ( 1 - a^*).
	\end{align}
	In fact, the conclusion of~\eqref{Eq:Akappa} remains true (trivially) when $x \in \mathcal{S}$.  Thus, by Proposition~\ref{prop:prelim}, 
	\begin{align*}
	R(C) - R(C^{\mathrm{Bayes}}) 
	&\leq \inf_{\kappa > 0} \Bigl\{\kappa \{\tilde{R}(C) - \tilde{R}(\tilde{C}^{\mathrm{Bayes}})\} + P_X(A_{\kappa}^c) \Bigr\}
	\\ & \leq \frac{\tilde{R}(C) - \tilde{R}(\tilde{C}^{\mathrm{Bayes}})}{  (1 - 2\rho^*) ( 1 - a^*)}  + P_X(A_{(1 - 2\rho^*)^{-1} ( 1 - a^*)^{-1}}^c) 
	=  \frac{\tilde{R}(C) - \tilde{R}(\tilde{C}^{\mathrm{Bayes}})}{  (1 - 2\rho^*) ( 1 - a^*)},
	\end{align*} 
	since $P_X(A_{(1 - 2\rho^*)^{-1} ( 1 - a^*)^{-1}}^c) \leq P_X(A_{(1 - 2\rho^*)^{-1} ( 1 - a^*)^{-1}}^c \cap \mathcal{B}) + P_X(\mathcal{B}^c) = 0$, by~\eqref{Eq:Akappa}.
\end{proof} 

\bigskip 

Proposition~\ref{prop:prelim} is a special case of the following result.
\begin{proposition}
	\label{prop:prelim2}
	Let $\mathcal{D} = \bigl\{x \in \mathcal{S}^c : \tilde{C}^{\mathrm{Bayes}}(x) = C^{\mathrm{Bayes}}(x)\bigr\}$, and recall the definition of $A_\kappa$ in Proposition~\ref{prop:prelim}.  Then, for any classifier $C$, 
	\begin{align*}
	R(C) - R(C^{\mathrm{Bayes}})  &\leq   R(\tilde{C}^{\mathrm{Bayes}}) - R(C^{\mathrm{Bayes}}) + \min\Bigl[\mathbb{P}\bigl\{ \{C(X) \neq  \tilde{C}^{\mathrm{Bayes}}(X)\} \cap \{X \in \mathcal{D}\} \bigr\}, 
	\\ & \hspace{80pt}   \inf_{\kappa > 0} \bigl\{\kappa \{\tilde{R}(C) - \tilde{R}(\tilde{C}^{\mathrm{Bayes}})\} + \mathbb{E}\bigl(|2\eta(X) - 1| \mathbbm{1}_{\{X \in \mathcal{D} \setminus A_{\kappa}\}}\bigr) \bigr\} \Bigr].
	\end{align*} 
\end{proposition}
Remark: If \eqref{eq:symmetric} holds, i.e.~$P_X(\mathcal{D}^c \cap \mathcal{S}^c) = 0$, then $R(\tilde{C}^{\mathrm{Bayes}}) = R(C^{\mathrm{Bayes}})$, and moreover we have that $\mathbb{E}\bigl(|2\eta(X) - 1| \mathbbm{1}_{\{X \in \mathcal{D} \setminus A_{\kappa}\}}\bigr) \leq P_X\bigl(\mathcal{D} \setminus A_{\kappa}\bigr) \leq P_X\bigl(A_{\kappa}^c\bigr)$.
\begin{proof}[Proof of Proposition~\ref{prop:prelim2}]
	First write
	\begin{align}
	\label{eq:risk}
	R(C) &=  \int_{\mathcal{X}} \mathbb{P}\{C(x) \neq Y \mid X=x\} \, dP_{X}(x) \nonumber
	\\& = \int_{\mathcal{X}} \bigl[ \mathbb{P}\{C(x) = 0 \}  \mathbb{P}(Y =1 \mid X = x) +  \mathbb{P}\{C(x) = 1 \}\mathbb{P}(Y =0\mid X = x) \bigr]\, dP_{X}(x) \nonumber
	\\& = \int_{\mathcal{X}} \bigl[ \mathbb{P}\{C(x) = 0 \}\{2\eta(x) - 1\}  + \{1-\eta(x)\} \bigr] \, dP_{X}(x).
	\end{align}
	Here we have implicitly assumed that the classifier $C$ is random since it may depend on random training data.  However, in the case that $C$ is non-random, one should interpret $\mathbb{P}\{C(x) = 0\}$ as being equal to $\mathbbm{1}_{\{C(x) = 0\} }$, for $x \in \mathcal{X}$. 
	
	Now, for $P_X$-almost all $x \in \mathcal{D}$, 
	\begin{align*}
	\bigl[\mathbb{P}\{C(x) = 0\} - \mathbbm{1}_{\{\tilde{\eta}(x) < 1/2\}}\bigr] \{2\eta(x) - 1\} & = \bigl|\mathbb{P}\{C(x) = 0\} - \mathbbm{1}_{\{\tilde{\eta}(x) < 1/2\}}\bigr| |2\eta(x) - 1|  
	\\ & \leq \bigl|\mathbb{P}\{C(x) = 0\} - \mathbbm{1}_{\{\tilde{\eta}(x) < 1/2\}}\bigr| 
	\\ & = \mathbb{P}\{C(x) \neq \tilde{C}^{\mathrm{Bayes}}(x)\}.
	\end{align*} 
	Moreover, for $P_X$-almost all $x \in \mathcal{D}^c$, we have
	\begin{equation}
	\label{Eq:NegativeProduct}
	\bigl[\mathbb{P}\{C(x) = 0\} - \mathbbm{1}_{\{\tilde{\eta}(x) < 1/2\}}\bigr] \{2\eta(x) - 1\} \leq 0
	\end{equation}
	It follows that 
	\begin{align*}
	R(C) - R(\tilde{C}^{\mathrm{Bayes}}) & = \int_{\mathcal{X}} \bigl[\mathbb{P}\{C(x) = 0\} - \mathbbm{1}_{\{\tilde{\eta}(x) < 1/2\}}\bigr] \{2\eta(x) - 1\} \, dP_{X}(x)
	\\ & = \int_{\mathcal{D}} \bigl[\mathbb{P}\{C(x) = 0\} - \mathbbm{1}_{\{\tilde{\eta}(x) < 1/2\}}\bigr] \{2\eta(x) - 1\} \, dP_{X}(x) 
	\\ & \hspace {30 pt} +  \int_{\mathcal{D}^c} \bigl[\mathbb{P}\{C(x) = 0\} - \mathbbm{1}_{\{\tilde{\eta}(x) < 1/2\}}\bigr] \{2\eta(x) - 1\} \, dP_{X}(x)
	\\ & \leq \mathbb{P}\bigl(\{C(X) \neq \tilde{C}^{\mathrm{Bayes}}(X)\} \cap \{X \in \mathcal{D}\} \bigr). 
	\end{align*}
	
	To see the right-hand bound, observe that by~\eqref{Eq:NegativeProduct}, for $\kappa > 0$, 
	\begin{align*} 
	R(C) - R(\tilde{C}^{\mathrm{Bayes}}) & = \int_{\mathcal{X}} \bigl[\mathbb{P}\{C(x) = 0\} - \mathbbm{1}_{\{\tilde{\eta}(x) < 1/2\}}\bigr] \{2\eta(x) - 1\} \, dP_{X}(x) 
	\\ & \leq \int_{\mathcal{D}} \bigl[\mathbb{P}\{C(x) = 0\} - \mathbbm{1}_{\{\tilde{\eta}(x) < 1/2\}}\bigr] \{2\eta(x) - 1\} \, dP_{X}(x)  
	\\ & \leq \kappa \int_{\mathcal{D} \cap A_\kappa} \bigl[\mathbb{P}\{C(x) = 0\} - \mathbbm{1}_{\{\tilde{\eta}(x) < 1/2\}}\bigr] \{2\tilde{\eta}(x) - 1\} \, dP_{X}(x) 
	\\ & \hspace{180 pt} + \mathbb{E}\bigl(|2\eta(X) - 1| \mathbbm{1}_{\{X \in \mathcal{D} \setminus A_{\kappa}\}}\bigr) 
	\\ & = \kappa \{\tilde{R}(C) - \tilde{R}(\tilde{C}^{\mathrm{Bayes}})\} + \mathbb{E}\bigl(|2\eta(X) - 1| \mathbbm{1}_{\{X \in \mathcal{D} \setminus A_{\kappa}\}}\bigr),
	\end{align*} 
	where the last step follows from~\eqref{eq:risk}.
\end{proof}

\bigskip 
\begin{example} 
	\label{ex:1nn}
\emph{Suppose that $\mathcal{X} \subseteq \mathbb{R}^d$ and that the noise is $\rho$-homogeneous with $\rho \in (0,1/2)$.  Consider the corrupted $1$-nearest neighbour classifier $\tilde{C}^{1\mathrm{nn}}(x) = \tilde{Y}_{(1)}$, where $(X_{(1)},\tilde{Y}_{(1)}) = (X_{(1)}(x),\tilde{Y}_{(1)}(x)) = (X_{i^*},\tilde{Y}_{i^*})$ is the training data pair for which $i^* = \sargmin_{i =1,\ldots,n} \|X_i-x\|$, where $\sargmin$ denotes the smallest index of the set of minimizers.  We first study the first term in the minimum in \eqref{eq:minbound}.  Noting that $\tilde{R}(\tilde{C}^{\mathrm{Bayes}}) = \mathbb{E}[\min\{\tilde{\eta}(X), 1-\tilde{\eta}(X) \}]$, we have }
	\begin{align}
	\label{Eq:LongDisplay}
	\bigl|\mathbb{P}\{\tilde{C}^{1\mathrm{nn}}&(X) \neq \tilde{C}^{\mathrm{Bayes}}(X)\} - \tilde{R}(\tilde{C}^{\mathrm{Bayes}}) \bigr| \nonumber
	\\ & = \bigl|\mathbb{P}\{\tilde{Y}_{(1)}(X)  \neq \tilde{C}^{\mathrm{Bayes}}(X)\} - \tilde{R}(\tilde{C}^{\mathrm{Bayes}})\bigr| \nonumber
	\\ & = \bigl| \mathbb{E}[\mathbbm{1}_{\{\tilde{\eta}(X) < 1/2\}}\tilde{\eta}(X_{(1)}(X)) + \mathbbm{1}_{\{\tilde{\eta}(X) \geq 1/2\}}\{1- \tilde{\eta}(X_{(1)}(X))\} ]  - \tilde{R}(\tilde{C}^{\mathrm{Bayes}})\bigr| \nonumber
	\\ & = \bigl|\mathbb{E}[\mathbbm{1}_{\{\tilde{\eta}(X) < 1/2\}}\{\tilde{\eta}(X_{(1)}(X)) - \tilde{\eta}(X) \} + \mathbbm{1}_{\{\tilde{\eta}(X) \geq 1/2\}}\{\tilde{\eta}(X)- \tilde{\eta}(X_{(1)}(X))\} ] \bigr| \nonumber
	\\& \leq  \mathbb{E}\bigl|\tilde{\eta}(X_{(1)}(X)) - \tilde{\eta}(X)\bigr | \rightarrow 0,
	\end{align}
\emph{	where the final limit follows by \citet[Lemma~5.4]{PTPR:1996}.}
	
\emph{	Now focusing on the second term in the minimum in~\eqref{eq:minbound}, by \citet[Theorem~5.1]{PTPR:1996}, we have}
	\[
	\tilde{R}(\tilde{C}^{1\mathrm{nn}}) - \tilde{R}(\tilde{C}^{\mathrm{Bayes}})  \rightarrow 2\mathbb{E}[\tilde{\eta}(X)\{1-\tilde{\eta}(X)\}] - \tilde{R}(\tilde{C}^{\mathrm{Bayes}}).
	\]
\emph{Moreover, in this case, $P_X(A_{\kappa}^c) = 1$ for all $\kappa \leq (1 - 2\rho)^{-1}$, and 0 otherwise.  Therefore, if $\rho$ is small enough that $\rho\tilde{R}(\tilde{C}^{\mathrm{Bayes}}) <  \tilde{R}(\tilde{C}^{\mathrm{Bayes}}) - \mathbb{E}[\tilde{\eta}(X)\{1-\tilde{\eta}(X)\}]$, then }
	\begin{align}
	\label{Eq:TwoMins}
	\lim_{n \rightarrow \infty} \inf_{\kappa > 0} \Bigl\{ \kappa \{\tilde{R}(\tilde{C}^{1\mathrm{nn}}) &- \tilde{R}(\tilde{C}^{\mathrm{Bayes}})\} + P_X(A_\kappa^c) \Bigr\} = \lim_{n\rightarrow \infty} \frac{\tilde{R}(\tilde{C}^{1\mathrm{nn}}) - \tilde{R}(\tilde{C}^{\mathrm{Bayes}})}{1 - 2\rho} \nonumber \\
	&= \frac{2\mathbb{E}[\tilde{\eta}(X)\{1-\tilde{\eta}(X)\}] - \tilde{R}(\tilde{C}^{\mathrm{Bayes}})}{1-2\rho} \nonumber \\
	&< \tilde{R}(\tilde{C}^{\mathrm{Bayes}}) = \lim_{n \rightarrow \infty} \mathbb{P}\{\tilde{C}^{1\mathrm{nn}}(X) \neq \tilde{C}^{\mathrm{Bayes}}(X)\},
	\end{align}
\emph{	where the final equality is due to~\eqref{Eq:LongDisplay}.  Thus, in this case, the second term in the minimum in~\eqref{eq:minbound} is smaller for sufficiently large $n$.  However, if $\rho\tilde{R}(\tilde{C}^{\mathrm{Bayes}}) > \tilde{R}(\tilde{C}^{\mathrm{Bayes}}) - \mathbb{E}[\tilde{\eta}(X)\{1-\tilde{\eta}(X)\}]$, the  asymptotically better bound is given by the first term in the minimum in the conclusion of Proposition~\ref{prop:prelim}, because then the inequality in~\eqref{Eq:TwoMins} is reversed.}
	\hfill $\Box$
\end{example} 
\bigskip 

\begin{proof}[Proof of Corollary~\ref{cor:consistent}]
	Let $\epsilon_n = \max\bigl[\sup_{m\geq n}\{\tilde{R}(\tilde{C}_m) - \tilde{R}(\tilde{C}^{\mathrm{Bayes}})\}^{1/2}, n^{-1}\bigr]$.  Then, by Proposition~\ref{prop:prelim2},
	\begin{align*}
	R(\tilde{C}_n) - R(\tilde{C}^{\mathrm{Bayes}}) &\leq \frac{1}{\epsilon_n} \{\tilde{R}(\tilde{C}_n) - \tilde{R}(\tilde{C}^{\mathrm{Bayes}})\} + \mathbb{E}\bigl(|2\eta(X) - 1| \mathbbm{1}_{\{X \in \mathcal{D} \setminus A_{\epsilon_n^{-1}}\}}\bigr)
	\\ & \leq  \{\tilde{R}(\tilde{C}_n) - \tilde{R}(\tilde{C}^{\mathrm{Bayes}})\}^{1/2} + P_X(\mathcal{D} \setminus A_{\epsilon_n^{-1}}\bigr).
	\end{align*}
	Since $(\epsilon_n)$ is decreasing, it follows that
	\[
	\limsup_{n \rightarrow \infty} R(\tilde{C}_n) - R(C^{\mathrm{Bayes}})  \leq R(\tilde{C}^{\mathrm{Bayes}}) - R(C^{\mathrm{Bayes}}) + P_X(\tilde{\mathcal{S}} \cap \mathcal{D}).
	\]
	In particular, if \eqref{eq:symmetric} holds, then   
	\[
	\limsup_{n \rightarrow \infty} R(\tilde{C}_n) - R(C^{\mathrm{Bayes}})  \leq P_X(\tilde{\mathcal{S}} \setminus \mathcal{S}),
	\]
	as required.
\end{proof}
\bigskip

\subsection{Conditions and proof of Theorem~\ref{thm:knnhet}}
\label{sec:knnproofs}
A formal description of the conditions of Theorem~\ref{thm:knnhet} is given below:
	\begin{description}
		\item[\normalfont{\emph{Assumption} A1.}] The probability measures $P_0$ and $P_1$ are absolutely continuous with respect to Lebesgue measure, with Radon--Nikodym derivatives $f_0$ and $f_1$, respectively. Moreover, the marginal density of $X$, given by $\bar{f} = \pi_{0} f_{0} + \pi_{1} f_{1}$, is continuous and positive.
\end{description}
	\begin{description}
		\item[\normalfont{\emph{Assumption} A2.}] The set $\mathcal{S}$ is non-empty and $\bar{f}$ is bounded on $\mathcal{S}$.  There exists $\epsilon_0 > 0$ such that $\bar{f}$ is twice continuously differentiable on $\mathcal{S}^{\epsilon_0} = \mathcal{S} + B_{\epsilon_0}(0)$, and
	\begin{equation}
	\label{eq:A2eq}
	F(\delta) = \sup_{x_0 \in \mathcal{S}: \bar{f}(x_0) \geq \delta} \max \biggl\{ \frac{\|\dot{\bar{f}}(x_0)\|}{\bar{f}(x_0)}, \frac{\sup_{u \in B_{\epsilon_0}(0)}\|\ddot{\bar{f}}(x_0+u)\|_{\mathrm{op}}}{ \bar{f}(x_0)} \biggr\} = o(\delta^{-\tau})
	\end{equation}
	as $\delta \searrow 0$, for every $\tau>0$.  Furthermore, recalling $a_d = \pi^{d/2}/\Gamma(1+d/2)$ and writing $p_\epsilon(x) = P_X(B_\epsilon(x))$, there exists $\mu_{0} \in (0,a_d)$ such that for all $x \in \mathbb{R}^{d}$ and $\epsilon \in (0,\epsilon_{0}]$, we have
	\[
	p_\epsilon(x) \geq \mu_{0} \epsilon^{d} \bar{f}(x).
	\]
\end{description}
	\begin{description}
		\item[\normalfont{\emph{Assumption} A3.}] We have $\inf_{x_0\in \mathcal{S}}\|\dot{\eta}(x_0)\| > 0$, so that $\mathcal{S}$ is a $(d-1)$-dimensional, orientable manifold.  Moreover, $\sup_{x \in \mathcal{S}^{2\epsilon_0}} \|\dot{\eta}(x)\| < \infty$ and $\ddot{\eta}$ is uniformly continuous on $\mathcal{S}^{2\epsilon_0}$ with $\sup_{x\in \mathcal{S}^{2\epsilon_0}} \|\ddot{\eta} (x)\|_{\mathrm{op}} < \infty$.  Finally, the function $\eta$ is continuous, and 
	\[
	\inf_{x \in \mathbb{R}^d \setminus \mathcal{S}^{\epsilon_0}} |\eta(x) - 1/2| > 0.
	\]
\end{description} 
	\begin{description}
		\item[\normalfont{\emph{Assumption} A4($\alpha$).}]
	We have that $\int_{\mathbb{R}^d}  \|x\|^{\alpha} \, dP_{X}(x) < \infty$ and $\int_{\mathcal{S}}  \bar{f}(x_0)^{d/(\alpha+d)} \, d\mathrm{Vol}^{d-1}(x_0) < \infty$, where $d\mathrm{Vol}^{d-1}$ denotes the $(d-1)$-dimensional volume form on $\mathcal{S}$. 
\end{description}

\begin{proof}[Proof of Theorem~\ref{thm:knnhet}]
	\textit{Part 1:} We show that the distribution $\tilde{P}$ of the pair $(X, \tilde{Y})$ satisfies suitably modified versions of Assumptions A1, A2, A3 and A4($\alpha$).
	
	Assumption A1: For $r \in \{0,1\}$, let $\tilde{P}_r$ denote the conditional distribution of $X$ given $\tilde{Y} = r$.  For $x \in \mathbb{R}^{d}$, and $r = 0,1$, define
	\[
	\tilde{f}_{r}(x)  =  \frac{\pi_{r}\{1-\rho_r(x)\}f_{r}(x) + \pi_{1-r}\rho_{1-r}(x) f_{1-r}(x)} {\int_{\mathbb{R}^{d}} \pi_{r}\{1-\rho_r(z)\}f_{1-r}(z) + \pi_{1-r}\rho_{1-r}(z)f_{1-r}(z) \, dz}.
	\]
	Now, for a Borel subset $A$ of $\mathbb{R}^{d}$, we have that
	\begin{align*} 
	\tilde{P}_{1}(A)  & = \mathbb{P}(X \in A \mid \tilde{Y} = 1) = \frac{\mathbb{P}(X \in A ,\tilde{Y} = 1)}{ \mathbb{P} ( \tilde{Y} =1) } 
	\\& =  \frac{\pi_{1}\mathbb{P}(X \in A ,\tilde{Y} = 1 \mid Y = 1) + \pi_{0}\mathbb{P}(X \in A ,\tilde{Y} = 1 \mid Y = 1)}{ \mathbb{P} ( \tilde{Y} =1) } 
	\\& = \frac{\pi_{1}\int_{A}  \{1-\rho_1(x)\}  f_{1}(x)  \, dx + \pi_{0} \int_{A}  \rho_0(x)  f_{0}(x)  \, dx } { \mathbb{P} ( \tilde{Y} =1) }  = \int_{A} \tilde{f}_{1}(x) \, dx. 
	\end{align*} 
	Similarly, $\tilde{P}_{0}(A) = \int_{A} \tilde{f}_{0}(x) \, dx$.  Hence $\tilde{P}_{0}$ and $\tilde{P}_{1}$ are absolutely continuous with respect to Lebesgue measure, with Radon--Nikodym derivatives $\tilde{f}_{0}$ and $\tilde{f}_{1}$, respectively.  Furthermore, $\tilde{f} = \mathbb{P}(\tilde{Y}=0) \tilde{f}_{0} +  \mathbb{P}(\tilde{Y}=1)\tilde{f}_{1} = \bar{f}$ is continuous and positive. 
	
	Assumption A2: Since A2 refers mainly to the marginal distribution of $X$, which is unchanged under the addition of label noise, this assumption is trivially satisfied for $\tilde{f} = \bar{f}$, as long as $\tilde{\mathcal{S}} = \{ x\in \mathbb{R}^{d} : \tilde{\eta}(x) = 1/2\} = \mathcal{S}$.  To see this, let $\delta_0 > 0$ and note that for $x$ satisfying $\eta(x) -1/2 > \delta_0$, we have from~\eqref{eq:tildeeta} that
	\begin{align}
	\label{eq:etabound}
	\tilde{\eta}(x) - 1/2 
	& =  \{\eta(x) - 1/2\}\{1 - \rho_0(x) - \rho_1(x)\}\Bigl\{1  + \frac{\rho_{0}(x) - \rho_{1}(x)}{\{2\eta(x)-1\}\{1 - \rho_0(x) - \rho_1(x)\}}\Bigr\}\nonumber
	\\& >  \{\eta(x) - 1/2\}(1- 2\rho^*)(1 - a^*) \geq   \delta_0 (1- 2\rho^*)(1 - a^*). 
	\end{align} 
	Similarly, if $1/2 - \eta(x) > \delta_0$, then we have that $1/2 - \tilde{\eta}(x) >  \delta_0 (1- 2\rho^*)(1 - a^*)$.  It follows that $\tilde{\mathcal{S}} \subseteq \mathcal{S}.$  Now, for $x$ such that $|\eta(x) -1/2| < \delta$, we have 
	\begin{equation}
	\label{Eq:geq}
	\tilde{\eta}(x) - 1/2 = \eta(x) - 1/2 + \{1-\eta(x)\}g(\eta(x)) - \eta(x)g(1-\eta(x)). 
	\end{equation}
	Thus $\mathcal{S} \subseteq \tilde{\mathcal{S}}$. 
	
	Assumption A3: Since $g$ is twice continuously differentiable, we have that $\tilde{\eta}$ is twice continuously differentiable on the set $\{x \in \mathcal{S}^{2\epsilon_0} : |\eta(x) -1/2| < \delta \}$.  On this set, its gradient vector at $x$ is
	\begin{align*} 
	\dot{\tilde{\eta}}(x) 
	& = \dot{\eta}(x)\Bigl[1 - g(\eta(x)) - g(1 - \eta(x)) + \{1-\eta(x)\} \dot{g}(\eta(x)) + \eta(x) \dot{g}(1 - \eta(x)) \Bigr].
	\end{align*} 
	The corresponding Hessian matrix at $x$ is 
	\begin{align*}
	\ddot{\tilde{\eta}}(x) & = \ddot{\eta}(x)\Bigl[1 - g(\eta(x)) - g(1 - \eta(x)) + \{1-\eta(x)\} \dot{g}(\eta(x)) + \eta(x) \dot{g}(1 - \eta(x))\Bigr ] 
	\\ & \hspace{30 pt} - \dot{\eta}(x) \Bigl[\dot{\eta}(x)^T \dot{g}(\eta(x)) - \dot{\eta}(x)^T \dot{g}(1 - \eta(x)) + \dot{\eta}(x)^T\dot{g}(\eta(x)) 
	\\ & \hspace{60 pt} - \{1-\eta(x)\} \dot{\eta}(x)^T \ddot{g}(\eta(x)) - \dot{\eta}(x)^T\dot{g}(1-\eta(x)) + \eta(x) \dot{\eta}(x)^T \ddot{g}(1 - \eta(x))\Bigr ].
	\end{align*} 
	In particular, for $x_0 \in \mathcal{S}$ we have 
	\begin{equation}
	\label{Eq:etadotddot}
	\dot{\tilde{\eta}}(x_0) = \dot{\eta}(x_0) \{1 - 2g(1/2) + \dot{g}(1/2)\}; \quad \ddot{\tilde{\eta}}(x_0) = \ddot{\eta}(x_0)\{1 - 2g(1/2) + \dot{g}(1/2)\}.
	\end{equation}
	Now define 
	\[
	\epsilon_1 = \sup\Bigl\{ \epsilon > 0 : \sup_{x \in \mathcal{S}^{2\epsilon}} |\eta(x) - 1/2| < \delta\Bigr\} > 0,
	\]
	where the fact that $\epsilon_1$ is positive follows from Assumption A3.  Set $\tilde{\epsilon}_0 = \min\{\epsilon_0, \epsilon_1\}/2$. Then, using the properties of $g$, we have that $\inf_{x_0\in \mathcal{S}}\|\dot{\tilde{\eta}}(x_0)\| > 0$.  Moreover, $\sup_{x \in \mathcal{S}^{2\tilde{\epsilon}_{0}}} \|\dot{\tilde{\eta}}(x)\| < \infty$ and $\ddot{\tilde{\eta}}$ is uniformly continuous on $\mathcal{S}^{2\tilde{\epsilon}_{0}}$ with $\sup_{x\in \mathcal{S}^{2\tilde{\epsilon}_{0}}} \|\ddot{\tilde{\eta}} (x)\|_{\mathrm{op}} < \infty$.  Finally, the function $\tilde{\eta}$ is continuous since $\rho_0,\rho_1$ are continuous, and, by~\eqref{eq:etabound},
	\[
	\inf_{x \in \mathbb{R}^d \setminus \mathcal{S}^{\tilde{\epsilon}_{0}}} |\tilde{\eta}(x) - 1/2| > 0.
	\]
	
	Assumption A4($\alpha$):  This holds for $\tilde{P}$ because the marginal distribution of $X$ is unaffected by the label noise and $\tilde{\mathcal{S}} = \mathcal{S}$.
	\bigskip
	
	\textit{Part 2}:  Recall the function $F$ defined in~\eqref{eq:A2eq}.  Let $c_n = F(k/(n-1))$, and set $\epsilon_n =  \{c_n\beta^{1/2} \log^{1/2}(n-1)\}^{-1}$, $\Delta_{n}= k(n-1)^{-1} c_{n}^{d}\log^{d}((n-1)/k)$,  $\mathcal{R}_n = \{x \in \mathbb{R}^d : \bar{f}(x) > \Delta_n \}$ and $\mathcal{S}_n= \mathcal{S} \cap \mathcal{R}_n$.  Then, by \eqref{eq:etabound} and the fact that $\inf_{x_0 \in \mathcal{S}} \|\dot{\tilde{\eta}}(x_0)\| > 0$, there exists $c_0 > 0$ such that for every $\epsilon \in (0,\tilde{\epsilon}_0]$, 
	\[
	\inf_{x \in \mathbb{R}^d \setminus \mathcal{S}^{\epsilon}} |\tilde{\eta}(x) -1/2| >   c_0\epsilon.
	\]
	Now let $\tilde{S}_{n}(x) = k^{-1} \sum_{i = 1}^{k} \mathbbm{1}_{\{\tilde{Y}_{(i)} = 1\}}$,  $X^{n} = (X_{1}, \ldots, X_{n})$ and $\tilde{\mu}(x, X^n) = \mathbb{E}\{\tilde{S}_{n}(x) \mid X^{n}\} = k^{-1}\sum_{i=1}^{k} \tilde{\eta}(X_{(i)})$. Define $A_{k} = \bigl\{\|X_{(k)}(x) - x\|  \leq \epsilon_{n}/2 \ \mbox{for all} \ x \in \mathcal{R}_{n} \bigr\}$.  Now suppose that $z_{1}, \dots, z_{N} \in \mathcal{R}_{n}$ are such that $\|z_{j} - z_{\ell}\| > \epsilon_{n}/4$ for all $j \neq \ell$,  but $\sup_{x \in \mathcal{R}_{n}} \min_{j=1, \dots, N} \|x - z_{j}\| \leq \epsilon_{n}/4$.  Then by the final part of Assumption A2, for $n \geq 2$ large enough that $\epsilon_n/8 \leq \epsilon_0$, we have
	\[
	1 = P_{X}(\mathbb{R}^{d}) \geq  \sum_{j=1}^{N} p_{\epsilon_{n}/8}(z_{j}) \geq \frac{N \mu_{0} \beta^{d/2}\log^{d/2}(n-1)}{8^d(n-1)^{1-\beta}}.
	\]
	Then by a standard binomial tail bound \citep[][Equation~(6), p.~440]{Shorack:86}, for such $n$ and any $M > 0$, 
	\begin{align*}
	\mathbb{P}(A_{k}^c) & = \mathbb{P}\Bigl\{ \sup_{x \in \mathcal{R}_{n} }\|X_{(k)}(x) - x\| > \epsilon_{n}/2\Bigr\} \nonumber \leq \mathbb{P}\Bigl\{\max_{j=1, \ldots, N}  \|X_{(k)}(z_{j}) - z_{j}\| > \epsilon_{n}/4 \Bigr\} \nonumber \\
	& \leq \sum_{j=1}^{N} \mathbb{P}\bigl\{\|X_{(k)}(z_{j}) - z_{j}\| > \epsilon_{n}/4 \bigr\} \nonumber  \leq N \max_{j=1,\ldots,N} \exp\Bigl(-\frac{1}{2}np_{\epsilon_n/4}(z_j) +k \Bigr)  = O(n^{-M}),
	\end{align*}
	uniformly for $k \in K_{\beta}$.
	
	Now, on the event $A_{k}$, for $\epsilon_n < \tilde{\epsilon}_0$ and $x \in \mathcal{R}_{n} \setminus \mathcal{S}^{\epsilon_{n}}$, the $k$ nearest neighbours of $x$ are on the same side of $\mathcal{S}$, so
	\begin{align*}
	|\tilde{\mu}_n(x,X^n) - 1/2| & = \biggl|\frac{1}{k} \sum_{i=1}^{k} \tilde{\eta}(X_{(i)}) - \frac{1}{2}\biggr|  \geq \inf_{z \in B_{\epsilon_n/2}(x)} |\tilde{\eta}(z) - 1/2| \geq c_0 \frac{\epsilon_n}{2}.
	\end{align*}
	Moreover, conditional on $X^n$, $\tilde{S}_n(x)$ is the sum of $k$ independent terms. Therefore, by Hoeffding's inequality,
	\begin{align}
	\label{eq:nM}
	\sup_{k \in K_\beta}& \sup_{x\in \mathcal{R}_{n}\setminus \mathcal{S}^{\epsilon_{n}}}  \bigl|\mathbb{P}\{\tilde{C}_{n}^{k\mathrm{nn}}(x) = 0\} -  \mathbbm{1}_{\{\tilde{\eta}(x) < 1/2\}}\bigr|    \nonumber
	\\  & = \sup_{k \in K_\beta} \sup_{x\in \mathcal{R}_{n}\setminus \mathcal{S}^{\epsilon_{n}}} \bigl|\mathbb{P}\{\tilde{S}_n(x) <1/2\} -  \mathbbm{1}_{\{\tilde{\eta}(x) < 1/2\}}\bigr|   \nonumber
	\\ & = \sup_{k \in K_\beta}\sup_{x\in \mathcal{R}_{n}\setminus \mathcal{S}^{\epsilon_{n}}} \bigl|\mathbb{E}\{ \mathbb{P}\{\tilde{S}_n(x) <1/2 \mid X^n\}- \mathbbm{1}_{\{\tilde{\eta}(x) < 1/2\}}\}\bigr| \nonumber
	\\ & \leq \sup_{k \in K_\beta}\sup_{x\in \mathcal{R}_{n}\setminus \mathcal{S}^{\epsilon_{n}}}\mathbb{E}\bigl[\exp(-2k \{ \tilde{\mu}_n(x, X^n) -1/2\}^2)\mathbbm{1}_{A_{k}} \bigr] + \sup_{k \in K_\beta}\mathbb{P}(A_{k}^c) = O(n^{-M})
	\end{align}
	for every $M > 0$.  
	
	Next, for $x \in \mathcal{S}^{\epsilon_2}$, we have $|\eta(x) - 1/2| < \delta$, and therefore, letting $t = \eta(x) - 1/2$, from~\eqref{Eq:geq} we can write
	\begin{align*}
	2\eta(x) -1 &- \frac{2\tilde{\eta}(x) -1}{1 - 2g(1/2) + \dot{g}(1/2)} 
	\\ & = \{2\eta(x) - 1\} \Bigl\{1 - \frac{1 - g(\eta(x)) - g(1-\eta(x))}{1 - 2g(1/2) + \dot{g}(1/2)} \Bigr\} - \frac{g(\eta(x)) - g(1 - \eta(x))}{1 - 2g(1/2) + \dot{g}(1/2)}
	\\ & =  2t \Bigl\{1 - \frac{1 - g(1/2+t) - g(1/2-t)}{1 - 2g(1/2) + \dot{g}(1/2)} \Bigr\} - \frac{g(1/2 + t) - g(1/2 - t)}{1 - 2g(1/2) + \dot{g}(1/2)} = G(t),
	\end{align*}
	say.  Observe that
	\[
	\dot{G}(t) =  2 \Bigl\{1 - \frac{1 - g(1/2+t) - g(1/2-t)}{1 - 2g(1/2) + \dot{g}(1/2)} \Bigr\} +  \frac{(2t-1) \dot{g}(1/2+t) - (2t+1) \dot{g}(1/2-t)}{1 - 2g(1/2) + \dot{g}(1/2)};
	\]
	and
	\[
	\ddot{G}(t) =  \frac{4\{\dot{g}(1/2+t) - \dot{g}(1/2-t)\}}{1 - 2g(1/2) + \dot{g}(1/2)} + \frac{(2t-1) \ddot{g}(1/2+t) + (2t+1) \ddot{g}(1/2-t)}{1 - 2g(1/2) + \dot{g}(1/2)}.
	\]
	In particular, we have $G(0) = 0$, $\dot{G}(0) = 0 $, $\ddot{G}(0) = 0$ and $\ddot{G}$ is bounded on $(-\delta,\delta)$. 
	
	Now there exists $n_0$ such that $\epsilon_n < \epsilon_2$, for all $n > n_0$ and $k \in K_\beta$.  Therefore, writing $\mathcal{S}_n^{\epsilon_n} = \mathcal{S}^{\epsilon_n} \cap \mathcal{R}_n$, for $n > n_0$, we have that
	\begin{align*} 
	&\Biggl|R(\tilde{C}^{k\mathrm{nn}}) - R(C^{\mathrm{Bayes}}) - \frac{\tilde{R}(\tilde{C}^{k\mathrm{nn}}) - \tilde{R}(\tilde{C}^{\mathrm{Bayes}}) }{1- 2g(1/2) + \dot{g}(1/2)} \Biggr | 
	\\& \hspace{0pt}  = \Biggl|\int_{\mathbb{R}^{d}} [\mathbb{P}\{\tilde{C}^{k\mathrm{nn}}(x) = 0\} - \mathbbm{1}_{\{\tilde{\eta}(x) < 1/2\}}] \Bigl\{2\eta(x) -1 - \frac{2\tilde{\eta}(x) -1}{1 - 2g(1/2) + \dot{g}(1/2)} \Bigr\}\, dP_X(x)\Biggr|
	\\ & \hspace{0pt}  \leq \Biggl|\int_{\mathcal{S}_n^{\epsilon_n}} [\mathbb{P}\{\tilde{C}^{k\mathrm{nn}}(x) = 0\} - \mathbbm{1}_{\{\tilde{\eta}(x) < 1/2\}}] \Bigl\{2\eta(x) -1 - \frac{2\tilde{\eta}(x) -1}{1 - 2g(1/2) + \dot{g}(1/2)} \Bigr\}\, dP_X(x)\Biggr|  
	\\ & \hspace{60 pt} + \biggl(1+ \frac{1}{1 - 2g(1/2) + \dot{g}(1/2)}\biggr)P_X(\mathcal{R}_n^c) + O(n^{-M}),
	\end{align*} 
	uniformly for $k \in K_\beta$, where the final claim uses~\eqref{eq:nM}.  Then, by a Taylor expansion of $G$ about $t = 0$, we have that
	\begin{align*} 
	&\Biggl|\int_{\mathcal{S}_n^{\epsilon_n}} [\mathbb{P}\{\tilde{C}^{k\mathrm{nn}}(x) = 0\} - \mathbbm{1}_{\{\tilde{\eta}(x) < 1/2\}}] \Bigl\{2\eta(x) -1 - \frac{2\tilde{\eta}(x) -1}{1 - 2g(1/2) + \dot{g}(1/2)} \Bigr\}\, dP_X(x)\Biggr|
	\\& \leq \frac{1}{2}\sup_{t \in (-\delta, \delta)} |\ddot{G}(t)|  \int_{\mathcal{S}_n^{\epsilon_n}} |\mathbb{P}\{\tilde{C}^{k\mathrm{nn}}(x) = 0\} - \mathbbm{1}_{\{\tilde{\eta}(x) < 1/2\}}| \{2\eta(x) -1\}^{2}  \, dP_X(x)
	\\ & \leq \frac{1}{2} \sup_{t \in (-\delta, \delta)} |\ddot{G}(t)| \sup_{x \in \mathcal{S}_n^{\epsilon_n}} |2\eta(x) -1|   \int_{\mathcal{S}_n^{\epsilon_n}} \{\mathbb{P}\{\tilde{C}^{k\mathrm{nn}}(x) = 0\} - \mathbbm{1}_{\{\tilde{\eta}(x) < 1/2\}}\} \{2\eta(x) -1\}  \, dP_X(x)\
	\\ & \leq \frac{1}{2}\sup_{t \in (-\delta, \delta)} |\ddot{G}(t)| \sup_{x \in \mathcal{S}_n^{\epsilon_n}} |2\eta(x) -1| \{  R(\tilde{C}^{k\mathrm{nn}}) - R(C^{\mathrm{Bayes}}) \} 
	\\ & \leq \frac{1}{2}\sup_{t \in (-\delta, \delta)} |\ddot{G}(t)| \sup_{x \in \mathcal{S}_n^{\epsilon_n}} |2\eta(x) -1| \frac{\tilde{R}(\tilde{C}^{k\mathrm{nn}}) - \tilde{R}(\tilde{C}^{\mathrm{Bayes}}) }{(1-2\rho^*)(1-a^*)}  =  o\Bigl(\tilde{R}(\tilde{C}^{k\mathrm{nn}}) - \tilde{R}(\tilde{C}^{\mathrm{Bayes}}) \Bigr),
	\end{align*}
	uniformly for $k \in K_\beta$.  
	
	Finally, to bound $P_X(\mathcal{R}_n^c)$, we have by the moment condition in Assumption A4($\alpha$) and H\"older's inequality, that for any $u \in (0,1)$, and $v > 0$,
	\begin{align*} 
	P_X(\mathcal{R}_n^c) &=  \mathbb{P}\{ \bar{f}(X) \leq \Delta_n \}  \leq  (\Delta_n)^\frac{\alpha(1-u)}{\alpha+d} \int_{x: \bar{f}(x) \leq  \Delta_n}  \bar{f}(x)^{1-\frac{\alpha(1-u)}{\alpha+d}} \, dx \nonumber
	\\ &  \leq  (\Delta_n)^\frac{\alpha(1-u)}{\alpha+d} \Bigl\{\int_{\mathbb{R}^d} (1+\|x\|^\alpha)\bar{f}(x) \, dx\Bigr\}^{1- \frac{\alpha(1-u)}{\alpha+d}}  \nonumber
	\\ & \hspace{120 pt} \Bigl\{\int_{\mathbb{R}^d} \frac{1}{(1+\|x\|^\alpha)^\frac{d+\alpha u}{\alpha(1-u)} } \, dx \Bigr\}^\frac{\alpha(1-u)}{\alpha+d}  = o\biggl(\Bigl(\frac{k}{n}\Bigr)^{\frac{\alpha(1-u)}{\alpha+d} - v}\biggr), \nonumber
	\end{align*}
	uniformly for $k \in K_{\beta}$.  
	
	Since $u \in (0,1)$ was arbitrary, we have shown that, that for any $v > 0$,
	\[
	R(\tilde{C}^{k\mathrm{nn}}) - R(C^{\mathrm{Bayes}}) - \frac{\tilde{R}(\tilde{C}^{k\mathrm{nn}}) - \tilde{R}(\tilde{C}^{\mathrm{Bayes}}) }{1- 2g(1/2) + \dot{g}(1/2)} = o\biggl(\tilde{R}(\tilde{C}^{k\mathrm{nn}}) - \tilde{R}(\tilde{C}^{\mathrm{Bayes}}) + \Bigl(\frac{k}{n}\Bigr)^{\frac{\alpha}{\alpha+d} - v}\biggr),
	\]
	uniformly for $k \in K_\beta$.   Since Assumptions A1,  A2, A3 and  A4($\alpha$) hold for $\tilde{P}$, the proof is completed by an application of \citet[][Theorem~1]{CBS:2017}, together with~\eqref{Eq:etadotddot}.  
\end{proof}

\subsection{Proofs from Section~\ref{sec:SVM}}
\label{sec:SVMproofs}
Before presenting the proofs from this section, we briefly discuss measurability issues for the SVM classifier.  Since this is constructed by solving the minimization problem in~\eqref{eq:SVMopttilde}, it is not immediately clear that it is measurable.  It is convenient to let $\mathcal{C}_d$ denote the set of all measurable functions from $\mathbb{R}^d$ to $\{0,1\}$.  By \citet[Definition~6.2, Lemma~6.3 and Lemma~6.23]{Steinwart:2008}, we have that the function $\tilde{C}_n^{\mathrm{SVM}}: (\mathbb{R}^d \times \{0,1\})^n \rightarrow \mathcal{C}_d$ and the map from $(\mathbb{R}^d \times \{0,1\})^n \times \mathbb{R}^d$ to $\{0,1\}$ given by $\bigl((x_1,\tilde{y}_1),\ldots,(x_n,\tilde{y}_n),x\bigr) \mapsto \tilde{C}_n^{\mathrm{SVM}}(x)$ are measurable with respect to the universal completion of the product $\sigma$-algebras on $(\mathbb{R}^d \times \{0,1\})^n$ and $(\mathbb{R}^d \times \{0,1\})^n \times \mathbb{R}^d$, respectively.  We can therefore avoid measurability issues by taking our underlying probability space $(\Omega, \mathcal{F}, \mathbb{P})$ to be as follows: let $\Omega = (\mathbb{R}^d \times \{0,1\} \times \{0,1\})^{n+1}$, and $\mathcal{F}$ to be the universal completion of the product $\sigma$-algebra on $\Omega$.  Moreover, we let $\mathbb{P}$ denote the canonical extension of the product measure on $\Omega$.  The triples $(X_1, Y_1,\tilde{Y}_1), \ldots, (X_n, Y_n, \tilde{Y}_n),(X,Y,\tilde{Y})$ can be taken to be the coordinate projections of the $(n+1)$ components of $\Omega$.  

\begin{proof}[Proof of Theorem~\ref{thm:SVMhet}]
	We first aim to show that $\tilde{P}$ satisfies the margin assumption with parameter $\gamma_1$, and has geometric noise exponent $\gamma_2$.  For the first of these claims, by~\eqref{Eq:Akappa}, we have for all $t > 0$ that
	\begin{align*}
	P_{X}(\{x \in \mathbb{R}^{d} : 0 < |\tilde{\eta}(x) - 1/2| \leq t\}) & \leq P_{X}\bigl(\bigl\{x: 0 < |\eta(x) - 1/2|(1-2\rho^{*})(1-a^{*})\leq t\bigr\}\bigr) 
	\\ & \leq \frac{\kappa_{1}}{(1-2\rho^{*})^{\gamma_1}(1-a^{*})^{\gamma_1}}  t^{\gamma_1},
	\end{align*}
	as required;  see also the discussion in Section 3.9.1 of the 2015 Australian National University PhD thesis by M. van Rooyen (\url{https://openresearch-repository.anu.edu.au/handle/1885/99588}).    The proof of the second claim is more involved, because we require a bound on $|2\tilde{\eta}(x) - 1|$ in terms of $|2\eta(x) - 1|$.   We consider separately the cases where $|\eta(x)-1/2|$ is small and large, and for $r > 0$, define $\mathcal{E}_r = \{x \in \mathbb{R}^d:|\eta(x)-1/2| < r\}$.  For $x \in \mathcal{E}_{\delta} \cap \mathcal{S}^c$, we can write $t_0 = \eta(x) - 1/2 \in (-\delta,\delta)$, so that by~\eqref{Eq:geq} again,
	\begin{align}
	\label{Eq:Step1}
	2\tilde{\eta}(x) -1 &= \{2\eta(x) -1\} \Bigl\{ 1 - g(\eta(x)) - g(1-\eta(x))   + \frac{g(\eta(x)) - g(1 - \eta(x))}{2\eta(x) -1}\Bigr\} \nonumber \\
	&= \{2\eta(x) -1\}\biggl\{1 - g(1/2 + t_0) - g(1/2-t_0)   + \frac{g(1/2+t_0) - g(1/2 - t_0) }{ 2t_0}\biggr\}.
	\end{align}
	Now, by reducing $\delta > 0$ if necessary, and since $1 - 2g(1/2) + \dot{g}(1/2) > 0$ by hypothesis, we may assume that
	\begin{equation}
	\label{Eq:Step2}
	\biggl|1 - g(1/2 + t_0) - g(1/2-t_0)   + \frac{g(1/2+t_0) - g(1/2 - t_0) }{ 2t_0}\biggr| \leq 2\{1 - 2g(1/2) + \dot{g}(1/2)\}
	\end{equation}
	for all $t_0 \in [-\delta,\delta]$.  Moreover, for $x \in \mathcal{E}_{\delta}^c$, we have
	\begin{align}
	\label{Eq:Step3}
	\Bigl| \{2\eta(x) - 1\} & \{1 - \rho_0(x) - \rho_1(x)\} + \rho_{0}(x) - \rho_{1}(x) \Bigr| \nonumber
	\\ & =  |2\eta(x) - 1|  \biggl|1 - \rho_0(x) - \rho_1(x) + \frac{\rho_{0}(x) - \rho_{1}(x)}{2\eta(x)-1} \biggr| \nonumber
	\\ & \leq |2\eta(x) - 1| \biggl\{ 1 + \frac{|\rho_{0}(x) - \rho_{1}(x)|}{2\delta}\biggr\}  \leq  |2\eta(x) - 1| \Bigl( 1 + \frac{1}{2\delta_0}\Bigr).
	\end{align}
	Now that we have the required bounds on $|2\tilde{\eta}(x)-1|$, we deduce from~\eqref{Eq:Step1},~\eqref{Eq:Step2} and~\eqref{Eq:Step3} that
	\begin{align*}
	\int_{\mathbb{R}^{d}} |2\tilde{\eta}(x) &- 1| \exp\Bigl(-\frac{\tau_{x}^{2}}{t^2}\Bigr) \, dP_{X}(x) \\
	&= \int_{\mathbb{R}^{d}} \Bigl| \{2\eta(x) - 1\} \{1 - \rho_0(x) - \rho_1(x)\} + \rho_{0}(x) - \rho_{1}(x) \Bigr| \exp\Bigl(-\frac{\tau_{x}^{2}}{t}\Bigr) \, dP_{X}(x) \\
	&\leq \max\Bigl\{2 -  4g(1/2) + 2\dot{g}(1/2), 1 + \frac{1}{2\delta_{0}}\Bigr\}\int_{\mathbb{R}^{d}} |2\eta(x) - 1| \exp\Bigl(-\frac{\tau_{x}^{2}}{t}\Bigr) \, dP_{X}(x) \\
	&\leq \max\Bigl\{2 -  4g(1/2) + 2\dot{g}(1/2), 1 + \frac{1}{2\delta_{0}}\Bigr\}\kappa_2 t^{\gamma_2d},
	\end{align*}
	so $\tilde{P}$ does indeed have geometric noise exponent $\gamma_2$.  
	
	Now, for an arbitrary classifier $C$, let $\tilde{L}(C) = \tilde{P}\bigl(\{(x,y) \in \mathbb{R}^d \times \{0,1\}:C(x) \neq y\}\bigr)$ denote the test error.  The quantity $\tilde{ L}(\tilde{C}^{\mathrm{SVM}})$ is random because the classifier depends on the training data and the probability in the definition of $\tilde{L}(\cdot)$ is with respect to test data only.  It follows by \citet[Theorem 2.8]{Steinwart:2007} that, for all $\epsilon >0$, there exists $M>0$ such that for all $n \in \mathbb{N}$ and all $\tau \geq 1$,
	\[
	\mathbb{P}\Bigl(\tilde{ L}(\tilde{C}^{\mathrm{SVM}}) - \tilde{L}(\tilde{C}^{\mathrm{Bayes}}) > M\tau^2 n^{-\Gamma + \epsilon}\Bigr) \leq e^{-\tau}.
	\]
	We conclude by Theorem~\ref{thm:hetnoise}(ii) that 
	\begin{align*}
	R(\tilde{C}^{\mathrm{SVM}}&) - R(C^{\mathrm{Bayes}}) \leq \frac{\tilde{R}(\tilde{C}^{\mathrm{SVM}}) - \tilde{R}(\tilde{C}^{\mathrm{Bayes}})}{(1-2\rho^*)(1-a^*)} \\
	&= \frac{1}{(1-2\rho^*)(1-a^*)} \int_{0}^{\infty} \mathbb{P}\Bigl(\tilde{ L}(\tilde{C}^{\mathrm{SVM}}) - \tilde{L}(\tilde{C}^{\mathrm{Bayes}}) > u \Bigr) \, du
	\\ & = \frac{2M n^{-\Gamma + \epsilon}}{(1-2\rho^*)(1-a^*)}  \int_{0}^{\infty} \tau \mathbb{P} \Bigl(\tilde{ L}(\tilde{C}^{\mathrm{SVM}}) - \tilde{L}(\tilde{C}^{\mathrm{Bayes}}) > M\tau^2 n^{-\Gamma + \epsilon}  \Bigr) \, d\tau
	\\ & \leq \frac{2M n^{-\Gamma + \epsilon}}{(1-2\rho^*)(1-a^*)}\biggl\{\int_{0}^{1} \tau \, d\tau + \int_{1}^\infty \tau \exp(-\tau) \, d\tau\biggr\} = \frac{M n^{-\Gamma + \epsilon}}{(1-2\rho^*)(1-a^*)} \Bigl(1 + \frac{4}{e}\Bigr),
	\end{align*}
	as required.
\end{proof}

\subsection{Proofs from Section~\ref{sec:LDA}}

\begin{proof}[Proof of Lemma~\ref{lem:tilde}]
	Since, for homogeneous noise, the pair $(X,Y)$ and the noise indicator $Z$ are independent, we have $\mathbb{P}\{C(X) \neq Y \mid Z = r \} = \mathbb{P}\{C(X) \neq Y \}$, for $r = 0,1$.  It follows that 
	\begin{align*}
	\tilde{R}(C) \! = \! \mathbb{P}\{C(X) \neq \tilde{Y}\} & = \mathbb{P}(Z = 1) \mathbb{P}\{C(X) \neq Y \mid Z = 1 \} + \mathbb{P}(Z = 0) \mathbb{P}\{C(X) = Y \mid Z = 0\}
	\\ & = (1 - \rho)\mathbb{P}\{C(X) \neq Y  \} + \rho [1 - \mathbb{P}\{C(X) \neq Y\}]\\
	& = \rho + (1-2\rho)R(C).
	\end{align*} 
	Rearranging terms gives the first part of the lemma,  and the second part follows immediately.
\end{proof}

\bigskip

\begin{proof}[Proof of Theorem~\ref{thm:LDAhomo}]
	For $r \in \{0,1\}$, we have that ${\hat{\pi}}_r \stackrel{\mathrm{a.s.}}{\rightarrow}  (1-\rho) \pi_r + \rho \pi_{1-r} = (1-2\rho) \pi_r + \rho$.   Now, writing 
	\[
	{\hat{\mu}}_r  = \frac{n^{-1}\sum_{i = 1}^n X_i \mathbbm{1}_{\{\tilde{Y}_i = r\}} }{{\hat{\pi}}_r} = \frac{n^{-1}\sum_{i = 1}^n X_i \mathbbm{1}_{\{\tilde{Y}_i = r\}} (\mathbbm{1}_{\{Y_i = r\}}+ \mathbbm{1}_{\{Y_i = 1 - r\}})} {{\hat{\pi}}_r},
	\]
	we see that
	\[
	{\hat{\mu}}_r  \stackrel{\mathrm{a.s.}}{\rightarrow} \frac{(1-\rho) \pi_{r} \mu_r + \rho \pi_{1-r} \mu_{1-r}}{(1-\rho) \pi_r + \rho \pi_{1-r}}.
	\]  
	Hence 
	\begin{align*}
	{\hat{\mu}}_1 + {\hat{\mu}}_0 &\stackrel{\mathrm{a.s.}}{\rightarrow}  \frac{(1-\rho) \pi_{1} \mu_1 + \rho \pi_{0} \mu_{0}}{(1-\rho) \pi_1 + \rho \pi_{0}}+\frac{(1-\rho) \pi_{0} \mu_0 + \rho \pi_{1} \mu_{1}}{(1-\rho) \pi_0 + \rho \pi_{1}}
	\\ & = \mu_{1} \biggl\{\frac{(1-2\rho)^{2}\pi_0\pi_1 + 2\rho(1-\rho)\pi_{1} } {(1-2\rho)^{2} \pi_0\pi_1 + \rho (1-\rho)  }\biggr\} + \mu_{0} \biggl\{\frac{ (1-2\rho)^{2} \pi_0\pi_1 + 2\rho(1-\rho) \pi_0 }{(1-2\rho)^{2} \pi_0 \pi_1 + \rho(1-\rho) }\biggr\}.
	\end{align*}
	Moreover
	\begin{align*} 
	{\hat{\mu}}_1 - {\hat{\mu}}_0  &\stackrel{\mathrm{a.s.}}{\rightarrow}  \frac{(1-\rho) \pi_{1} \mu_1 + \rho \pi_{0} \mu_{0}}{(1-\rho) \pi_1 + \rho \pi_{0}} - \frac{(1-\rho) \pi_{0} \mu_0 + \rho \pi_{1} \mu_{1}}{(1-\rho) \pi_0 + \rho \pi_{1}}
	\\ & = \biggl\{\frac{(1-2\rho) \pi_{0} \pi_{1}} {(1-2\rho)^{2} \pi_0\pi_1 + \rho (1-\rho)  }\biggr\} (\mu_{1} - \mu_{0}).
	\end{align*}   
	Observe further that 
	\begin{align*}
	{\hat{\Sigma}} &\stackrel{\mathrm{a.s.}}{\rightarrow} \mathrm{cov}\bigl((X_1 - \tilde{\mu}_1)(X_1 - \tilde{\mu}_1)^T\mathbbm{1}_{\{\tilde{Y}_1=1\}} + (X_1 - \tilde{\mu}_0)(X_1 - \tilde{\mu}_0)^T\mathbbm{1}_{\{\tilde{Y}_1=0\}}\bigr) \\  
	&= \{(1-2\rho) \pi_1 + \rho\} \tilde{\Sigma}_1 +   \{(1-2\rho) \pi_0 + \rho\} \tilde{\Sigma}_0,
	\end{align*}
	where $\tilde{\Sigma}_r = \mathrm{cov}(X \mid \tilde{Y} = r)$, and we now seek to express $\tilde{\Sigma}_0$ and $\tilde{\Sigma}_1$ in terms of $\rho$, $\pi_0$, $\pi_1$, $\mu_0$, $\mu_1$ and $\Sigma$.  To that end, we have that
	\[
	\tilde{\Sigma}_r = \mathbb{E}\{ \mathrm{cov}(X \mid Y, \tilde{Y} = r)\mid \tilde{Y} = r \} + \mathrm{cov}\{ \mathbb{E}( X \mid Y, \tilde{Y} = r) \mid  \tilde{Y} = r\}  = \Sigma +  \mathrm{cov}\{ \mu_Y \mid  \tilde{Y} = r\}.
	\]
	Note that 
	\[
	\mathbb{P}(Y =1 \mid  \tilde{Y} = 1) = \frac{ \mathbb{P}(Y =1 ,  \tilde{Y} = 1)}{ \mathbb{P}(\tilde{Y} = 1) } = \frac{\pi_1 (1-\rho)}{ \pi_1  (1-\rho) + \pi_0 \rho } =  \frac{\pi_1 (1-\rho)}{ \pi_1 (1 - 2\rho) + \rho }.
	\]
	Hence
	\[
	\mathbb{E}(\mu_Y \mid  \tilde{Y} = 1) = \mu_1 \mathbb{P}(Y =1 \mid  \tilde{Y} = 1) + \mu_0 \mathbb{P}(Y =0 \mid  \tilde{Y} = 1) = \frac{\pi_1 \mu_1 (1-\rho) + \pi_0 \mu_0 \rho}{ \pi_1 (1-2\rho) + \rho }.
	\]
	It follows that 
	\begin{align*}
	\tilde{\Sigma}_1 &= \frac{\pi_1 (1-\rho)}{ \pi_1 (1 - 2\rho) + \rho } \Bigl(\mu_1 -  \frac{\pi_1 \mu_1 (1-\rho) + \pi_0 \mu_0 \rho }{ \pi_1 (1 - 2\rho) + \rho } \Bigr) \Bigl(\mu_1 -  \frac{\pi_1 \mu_1 (1-\rho) + \pi_0 \mu_0 \rho }{ \pi_1 (1 - 2\rho) + \rho } \Bigr)^T 
	\\ & \hspace{20pt} +  \frac{\pi_0 \rho}{ \pi_1 (1 - 2\rho) + \rho } \Bigl(\mu_0 -  \frac{\pi_1 \mu_1 (1-\rho) + \pi_0 \mu_0 \rho }{ \pi_1 (1 - 2\rho) + \rho } \Bigr) \Bigl(\mu_0 -  \frac{\pi_1 \mu_1 (1-\rho) + \pi_0 \mu_0 \rho }{ \pi_1 (1 - 2\rho) + \rho } \Bigr)^T
	\\ &= \frac{\pi_1 (1-\rho)}{ \pi_1 (1 - 2\rho) + \rho } \Bigl( \frac{\pi_0 \rho (\mu_1 - \mu_0) }{ \pi_1 (1 - 2\rho) + \rho } \Bigr) \Bigl( \frac{\pi_0 \rho (\mu_1 - \mu_0) }{ \pi_1 (1 - 2\rho) + \rho } \Bigr)^T
	\\ &  \hspace{20pt}  +  \frac{\pi_0 \rho}{ \pi_1 (1 - 2\rho) + \rho } \Bigl(  \frac{\pi_1(1- \rho)(\mu_0 - \mu_1)}{ \pi_1 (1 - 2\rho) + \rho } \Bigr)\Bigl(  \frac{\pi_1(1- \rho)(\mu_0 - \mu_1)}{ \pi_1 (1 - 2\rho) + \rho } \Bigr)^T
	\\ & = \frac{\pi_0 \pi_1\rho(1-\rho)}{ (\pi_1 (1 - \rho) + \pi_0\rho)^2 }  (\mu_1 - \mu_0)(\mu_1 - \mu_0)^T.
	\end{align*} 
	Similarly
	\[
	\tilde{\Sigma}_0 = \frac{\pi_0\pi_1 \rho(1-\rho)}{ (\pi_0 (1 - \rho) + \pi_1\rho)^2 }  (\mu_1 - \mu_0)(\mu_1 - \mu_0)^T.
	\]
	We deduce that
	\[
	{\tilde{\Sigma}} \stackrel{\mathrm{a.s.}}{\rightarrow} \Sigma + \frac{\pi_0 \pi_1 \rho(1-\rho)}{\pi_1\pi_0(1-2\rho)^2 + \rho(1-\rho) } (\mu_1 - \mu_0)(\mu_1 - \mu_0)^T = \Sigma + \alpha (\mu_1 - \mu_0)(\mu_1 - \mu_0)^T,
	\]
	where $\alpha = \pi_0 \pi_1 \rho(1-\rho)/\{\pi_0\pi_1(1-2\rho)^2 + \rho(1-\rho)\}$.   Now
	\[
	\bigl(\Sigma + \alpha (\mu_1 - \mu_0)(\mu_1 - \mu_0)^T\bigr)^{-1} = \Sigma^{-1} - \frac{ \alpha \Sigma^{-1} (\mu_1 - \mu_0) (\mu_1 - \mu_0)^T \Sigma^{-1}}{1 + \alpha \Delta^2},
	\]
	where $\Delta^2 = (\mu_1 - \mu_0)^T\Sigma^{-1} (\mu_1 - \mu_0)$.   It follows that there exists an event $\Omega_0$ with $\mathbb{P}(\Omega_0) = 1$ such that on this event, for every $x \in \mathbb{R}^d$,
	\begin{align*}
	\Bigl(x - & \frac{{\hat{\mu}}_1 + {\hat{\mu}}_0}{2}\Bigr)^T {\hat{\Sigma}}^{-1} ({\hat{\mu}}_1 - {\hat{\mu}}_0) 
	\\&\rightarrow \biggl[x -   \frac{\mu_{1}}{2} \Bigl\{\frac{(1-2\rho)^{2}\pi_0\pi_1 + 2\rho(1-\rho)\pi_{1} } {(1-2\rho)^{2} \pi_0\pi_1 + \rho (1-\rho)  }\Bigr\} + \frac{\mu_{0}}{2} \Bigl\{\frac{ (1-2\rho)^{2} \pi_0\pi_1 + 2\rho(1-\rho) \pi_0 }{(1-2\rho)^{2} \pi_0 \pi_1 + \rho(1-\rho) }\Bigr\}\biggr]^T 
	\\& \hspace{30pt}  \Bigl(\Sigma^{-1} - \frac{ \alpha \Sigma^{-1} (\mu_1 - \mu_0) (\mu_1 - \mu_0)^T \Sigma^{-1}}{1 + \alpha \Delta^2} \Bigr) \Bigl\{\frac{(1-2\rho) \pi_0 \pi_1} {(1-2\rho)^{2} \pi_0\pi_1 + \rho (1-\rho)  }\Bigr\} (\mu_{1} - \mu_{0})
	\\& =  \biggl[x -   \frac{\mu_{1}}{2} \Bigl\{\frac{(1-2\rho)^{2}\pi_0\pi_1 + 2\rho(1-\rho)\pi_{1} } {(1-2\rho)^{2} \pi_0\pi_1 + \rho (1-\rho)  }\Bigr\} + \frac{\mu_{0}}{2} \Bigl\{\frac{ (1-2\rho)^{2} \pi_0\pi_1 + 2\rho(1-\rho) \pi_0 }{(1-2\rho)^{2} \pi_0 \pi_1 + \rho(1-\rho) }\Bigr\}\biggr]^T 
	\\& \hspace{60pt}  \Bigl(\frac{ 1}{1 + \alpha\Delta^{2}}\Bigr) \Bigl\{\frac{(1-2\rho) \pi_{0} \pi_{1}} {(1-2\rho)^{2} \pi_0\pi_1 + \rho (1-\rho)  }\Bigr\} \Sigma^{-1}(\mu_{1} - \mu_{0})
	\\& =  \biggl(x - \frac{\mu_{1} + \mu_{0}}{2}\biggr)^T  \Bigl(\frac{ 1}{1 + \alpha\Delta^{2}}\Bigr) \Bigl\{\frac{(1-2\rho) \pi_{0} \pi_{1}} {(1-2\rho)^{2} \pi_0\pi_1 + \rho (1-\rho)  }\Bigr\} \Sigma^{-1}(\mu_{1} - \mu_{0})
	\\ &  \hspace{30pt}  -  \biggl[\frac{\mu_{1}}{2} \Bigl\{\frac{(2\pi_{1} - 1)\rho(1-\rho)} {(1-2\rho)^{2} \pi_0\pi_1 + \rho (1-\rho)  }\Bigr\} + \frac{\mu_{0}}{2} \Bigl\{\frac{ (2\pi_{0} - 1)\rho(1-\rho)}{(1-2\rho)^{2} \pi_0 \pi_1 + \rho(1-\rho) }\Bigr\} \biggr]^T 
	\\& \hspace{ 60pt}  \Bigl(\frac{ 1}{1 + \alpha\Delta^{2}}\Bigr) \Bigl\{\frac{(1-2\rho) \pi_{0} \pi_{1}} {(1-2\rho)^{2} \pi_0\pi_1 + \rho (1-\rho)  }\Bigr\} \Sigma^{-1}(\mu_{1} - \mu_{0}).
	\end{align*}
	Hence, on $\Omega_0$,
	\[
	\lim_{n \rightarrow \infty} \tilde{C}^{\mathrm{LDA}}(x) = \left\{ \begin{array}{ll} 1  & \text{if } c_0 +  \bigl(x - \frac{\mu_1 + \mu_0}{2}\bigr)^T \Sigma^{-1} (\mu_1 - \mu_0) > 0
	\\ 0 & \text{if } c_0 + \bigl(x - \frac{\mu_1 + \mu_0}{2}\bigr)^T \Sigma^{-1} (\mu_1 - \mu_0) < 0, \end{array} \right.
	\]
	where 
	\begin{align*}
	c_0  & = \frac{(1 + \alpha \Delta^2)\rho(1-\rho)}{\alpha(1-2\rho)} \log\biggl(\frac{(1-2\rho)\pi_1 + \rho}{(1-2\rho)\pi_0 + \rho} \biggr) - \frac{(\pi_{1} - \pi_{0})\alpha\Delta^2}{2\pi_0\pi_1} .
	\end{align*}
	This proves the first claim of the theorem.  It follows that
	\begin{align*} 
	\lim_{n \rightarrow \infty} R(\tilde{C}^{\mathrm{LDA}}) & =  \pi_0\Phi\biggl(\frac{c_0}{\Delta} - \frac{\Delta}{2}\biggr) +  \pi_1\Phi\biggl(-\frac{c_0}{\Delta} - \frac{\Delta}{2}\biggr), 
	\end{align*}
	which proves the second claim.  Now consider the function
	\[
	\psi(c_0) = \pi_0\Phi\biggl(\frac{c_0}{\Delta} - \frac{\Delta}{2}\biggr) +  \pi_1\Phi\biggl(-\frac{c_0}{\Delta} - \frac{\Delta}{2}\biggr).
	\]
	We have
	\[
	\dot{\psi}(c_0) = \frac{\pi_0}{\Delta}\phi\biggl(\frac{c_0}{\Delta} - \frac{\Delta}{2}\biggr) -  \frac{\pi_1}{\Delta} \phi\biggl(-\frac{c_0}{\Delta} - \frac{\Delta}{2}\biggr) =  \frac{\pi_0}{\Delta}\phi\biggl(\frac{c_0}{\Delta} - \frac{\Delta}{2}\biggr) \Bigl\{1  -  \frac{\pi_1}{\pi_0} \exp(-c_0)\Bigr\},
	\]
	where $\phi$ denotes the standard normal density function.  Since $\mathrm{sgn}\bigl(\dot{\psi}(c_0)\bigr) = \mathrm{sgn}\bigl(c_0 - \log(\pi_1/\pi_0)\bigr)$, we deduce that
	\[
	\pi_0\Phi\biggl(\frac{c_0}{\Delta} - \frac{\Delta}{2}\biggr) +  \pi_1\Phi\biggl(-\frac{c_0}{\Delta} - \frac{\Delta}{2}\biggr) \geq R(C^{\mathrm{Bayes}}),
	\]
	and it remains to show that if $\rho \in (0,1/2)$ and $\pi_1 \neq \pi_0$, then there is a unique $\Delta > 0$ with $c_0 = \log(\pi_1/\pi_0)$.   To that end, suppose without loss of generality that $\pi_1 > \pi_0$ and note that
	\begin{align*}
	\frac{(\pi_{1} - \pi_{0})(1-2\rho)}{(1-2\rho)^{2} \pi_0 \pi_1 + \rho(1-\rho)} &= \frac{\pi_{1} (1-2\rho) + \rho }{(1-2\rho)^{2} \pi_1 \pi_{0} + \rho(1-\rho)} - \frac{\pi_{0} (1-2\rho) + \rho }{(1-2\rho)^{2} \pi_1 \pi_{0} + \rho(1-\rho)}   
	\\ &= \frac{1 }{(1-2\rho)\pi_{0} + \rho} - \frac{1}{(1-2\rho) \pi_1 + \rho}.
	\end{align*}
	Hence, writing $t = (1-2\rho)\pi_1 + \rho > 1/2$, we have
	\[ 
	\log\Bigl(\frac{(1-2\rho)\pi_1 + \rho}{(1-2\rho)\pi_0 + \rho} \Bigr) - \frac{(\pi_{1} - \pi_{0})(1-2\rho)}{2\{(1-2\rho)^{2} \pi_1 \pi_{0} + \rho(1-\rho)\}} = 
	\log\Bigl(\frac{t}{1-t} \Bigr) + \frac{1}{2t} - \frac{1}{2(1-t)} < 0.
	\]
	Next, let
	\begin{align*}
	\chi(\pi_1) &= \log\Bigl(\frac{\pi_1}{\pi_0}\Bigr) - \frac{\rho(1-\rho)}{\alpha(1-2\rho)} \log\Bigl(\frac{(1-2\rho)\pi_1 + \rho}{(1-2\rho)\pi_0 + \rho} \Bigr) \nonumber
	\\ &= \log\Bigl(\frac{\pi_1}{1 - \pi_1}\Bigr) - \frac{(1-2\rho)^{2}\pi_{1}(1-\pi_{1}) + \rho(1-\rho)}{(1-2\rho)\pi_1(1-\pi_1)} \log\Bigl(\frac{(1-2\rho)\pi_1 + \rho}{(1-2\rho)(1-\pi_1) + \rho} \Bigr).
	\end{align*}
	Then
	\[
	\dot{\chi}(\pi_1) = \frac{\rho(1-\rho) (1- 2\pi_1) }{(1-2\rho) \pi_1^2(1-\pi_1)^2} \log\Bigl(\frac{(1-2\rho)\pi_1 + \rho}{(1-2\rho)(1-\pi_1) + \rho} \Bigr) < 0, 
	\]
	for all $\pi_1 \in (0,1)$.  Since $\chi(1/2) = 0$, we conclude that $\chi(\pi_1) < 0$ for all $\pi_1 > \pi_0$.  But
	\[
	c_0 - \log\Bigl(\frac{\pi_1}{\pi_0}\Bigr) = \frac{\Delta^2 \rho(1-\rho)}{1-2\rho}  \Bigl\{\log\Bigl(\frac{(1-2\rho)\pi_1 + \rho}{(1-2\rho)\pi_0 + \rho} \Bigr) - \frac{(\pi_{1} - \pi_{0})(1-2\rho)}{2\{(1-2\rho)^{2} \pi_1 \pi_{0} + \rho(1-\rho)\}} \Bigr\} - \chi(\pi_1),
	\]
	so the final claim follows.
\end{proof} 

\section*{Acknowledgements}
The authors would like to thank Jinchi Lv for introducing us to this topic, and the anonymous reviewers for helpful and constructive comments. The second author is partly supported by NSF CAREER Award DMS-1150318, and the third author is supported by an EPSRC Fellowship EP/P031447/1 and grant RG81761 from the Leverhulme Trust.  The authors would like to thank the Isaac Newton Institute for Mathematical Sciences for support and hospitality during the programme `Statistical Scalability' when work on this paper was undertaken.  This work was supported by Engineering and Physical Sciences Research Council grant numbers EP/K032208/1 and EP/R014604/1.

\end{document}